\documentclass[12pt,a4paper]{article}

\usepackage[latin1]{inputenc}
\usepackage{amsmath}
\usepackage{amsfonts}
\usepackage{graphicx}
\usepackage{amsthm}
\usepackage{amssymb}
\usepackage{rotating}

\renewcommand{\baselinestretch}{1.3}



\newcommand{\bqa}{\begin{eqnarray*}}
\newcommand{\eqa}{\end{eqnarray*}}
\newcommand{\bqan}{\begin{eqnarray}}
\newcommand{\eqan}{\end{eqnarray}}
\newcommand{\bqt}{\begin{quote}}
\newcommand{\eqt}{\end{quote}}
\newcommand{\bt}{\begin{tabbing}}
\newcommand{\et}{\end{tabbing}}
\newcommand{\bit}{\begin{itemize}}
\newcommand{\eit}{\end{itemize}}
\newcommand{\ben}{\begin{enumerate}}
\newcommand{\een}{\end{enumerate}}
\newcommand{\beq}{\begin{equation}}
\newcommand{\eeq}{\end{equation}}
\newcommand{\bdefi}{\begin{definition}}
\newcommand{\edefi}{\end{definition}}
\newcommand{\bpro}{\begin{proposition}}
\newcommand{\epro}{\end{proposition}}
\newcommand{\blem}{\begin{lemma}}
\newcommand{\elem}{\end{lemma}}
\newcommand{\bth}{\begin{theorem}}
\newcommand{\bco}{\begin{corollary}}
\newcommand{\eco}{\end{corollary}}
\newcommand{\bdes}{\begin{description}}
\newcommand{\edes}{\end{description}}
\newcommand{\bre}{\begin{remark}}
\newcommand{\ere}{\end{remark}}

\newtheorem{definition}{Definition}
\newtheorem{proposition}{Proposition}
\newtheorem{lemma}{Lemma}
\newtheorem{theorem}{Theorem}
\newtheorem{corollary}{Corollary}
\newtheorem{remark}{Remark}

\newcommand{\ep}{\varepsilon}

\newcommand{\sumi}{\sum_{i=1}^n}
\newcommand{\sumj}{\sum_{j=1}^n}

\newcommand{\IR}{\mbox{$I\!\!R$}}

\renewenvironment{proof}[1][Proof]{\noindent\textbf{#1.} }{\qed}%{\\rule{0.5em}{0.5em}}
\renewcommand{\baselinestretch}{1.62}

\usepackage{verbatim,color,array}

\begin{document}

\def\baselinestretch{1.2}

\addtocounter{footnote}{3}
\author{ Enno Mammen\thanks{%
Institut f\"ur Angewandte Mathematik,
Universit\"at Heidelberg,
Im Neuenheimer Feld 205,
69120 Heidelberg,
Germany. E-mail address: \texttt{mammen@math.uni-heidelberg.de}.} \\
%EndAName
Heidelberg University, Germany %\\ and\\ National Research University\\ Higher School of Economics, Russian Federation
 \and
 \addtocounter{footnote}{2}
 Ingrid Van Keilegom
 \thanks{
ORSTAT, KU Leuven, Naamsestraat 69, 3000 Leuven, Belgium. E-mail address: \texttt{%
ingrid.vankeilegom@kuleuven.be}. } \\
KU Leuven, Belgium
 \and
  \addtocounter{footnote}{1}
 Kyusang Yu
 \thanks{
Department of Applied Statistics, Konkuk University, Seoul 143-701, Korea, E-mail address: \texttt{
kyusangu@konkuk.ac.kr
 }.
 }\\
%EndAName
Konkuk University, Seoul, Korea
 }
\title{\textbf{EXPANSION FOR MOMENTS OF REGRESSION QUANTILES WITH APPLICATIONS TO NONPARAMETRIC TESTING }}
%Expansion for moments of regression quantiles with applications to nonparametric testing
\maketitle

\begin{abstract}
We discuss nonparametric tests for parametric specifications of regression quantiles. The test is based on the comparison of parametric and nonparametric fits of these quantiles.  The nonparametric fit is a Nadaraya-Watson quantile smoothing estimator.

An asymptotic treatment of the test statistic requires the development of new mathematical arguments. An approach that makes only use of plugging in a Bahadur expansion of the nonparametric estimator is not satisfactory. It requires too strong conditions on the dimension and the choice of the bandwidth.

Our alternative mathematical approach requires the calculation of moments of Nadaraya-Watson quantile regression estimators. This calculation is done by application of higher order Edgeworth expansions. %The moments allow estimation bounds for the accuracy of Bahadur expansions for integrals of kernel quantile estimators.

%Another application of our method gives asymptotic results for the estimation of weighted averages of regression quantiles.
\end{abstract}


\noindent\textit{Journal of Economic Literature Classification}:\ C14

\vspace{0.1in}

\def\baselinestretch{1.3}

\section{Introduction}

Consider a data set of $n$ i.i.d.\ tuples $(X_i,Y_i)$, where $Y_i$ is a one-dimensional response variable and $X_i$ is a $d$-dimensional covariate. For $0<\alpha<1$ we denote the conditional $\alpha$-quantile of $Y_i$ given $X_i=x$ by  $m_\alpha(x)$.
Thus we can write
\begin{eqnarray}\label{model}
Y_i = m_\alpha(X_i) + \ep_{i,\alpha} \hspace*{.5cm} (i=1,\ldots,n),
\end{eqnarray}
with error variables $\ep_{i,\alpha}$ that fulfill $q_\alpha(\ep_{i,\alpha}|X_i) = 0$.  Here, $q_\alpha(\ep_{i,\alpha}|X_i)$ is the $\alpha$-quantile of the conditional distribution of $\ep_{i,\alpha}$ given $X_i$. Consider the null hypothesis
\begin{eqnarray} \label{H0}
H_0 : \mbox{For all } \alpha \in A \mbox{ there exists a } \theta(\alpha) \in \Theta, \mbox{ such that } m_\alpha = m_{\alpha,\theta(\alpha)},
\end{eqnarray}
where $\{m_{\alpha,\theta} : \theta \in \Theta\}$ is a parametric class of regression quantiles, $\Theta$ is a compact subset of $\IR^k$ and $A \subset (0,1)$.  The set $A$ can be a singleton $A=\{\alpha\}$, but can also be a closed subset of $(0,1)$ if a set of quantile functions is checked.

In this paper we aim at studying a test statistic for $H_0$, and to study its asymptotic properties under the null and the alternative.  We will see that this problem is an example of a quantile model where the asymptotics cannot be developed by standard tools of quantile regression. In particular, a direct application of Bahadur expansions requires assumptions that are too restrictive.

Our test statistic is based on kernel smoothing. Let $K(u_1,\ldots,u_d) = \prod_{j=1}^d k(u_j)$, where $k$ is a one-dimensional density function defined on $[-1,1]$, and let $h=(h_1,\ldots,h_d)$ be a $d$-dimensional bandwidth parameter.  We assume that all bandwidths $ h_1,\ldots,h_d$ are of the same order. For simplicity of notation we further assume that they are identical and by abuse of notation we write $h=  h_1= \ldots =h_d$. For any $0<\alpha<1$ and any $x$ in the support $R_X$ of $X$, let $F_{\varepsilon_\alpha|X}(\cdot|x) $ be the conditional distribution function of $\varepsilon_\alpha=Y-m_\alpha(X)$, given $X=x$, and let $r_{\alpha,\theta(\alpha)}(x)$ be the $\alpha$-quantile of $Y-m_{\alpha,\theta(\alpha)}(X)$ given that $X=x$. Define
$$ \widehat r_{\alpha}(x) = \arg \min_r \sum_{i=1}^n K\left ( {x-X_i \over h}\right ) \tau_{\alpha} (Y_i-m_{\alpha,\widehat \theta(\alpha)}(X_i)-r), $$
where $\tau_{\alpha} (u) = \alpha u_+ - (1-\alpha) u_-$, $u_+ = uI(u>0)$ and $u_- = uI(u<0)$ and where $\widehat \theta(\alpha)$ is an estimator of $\theta(\alpha)$. 

Note that instead of estimating the conditional quantile $r_{\alpha,\theta(\alpha)}(x)$ by the above estimator, we could have considered the alternative estimator
$$ \widehat r_{\alpha}^{alt}(x) = \arg \min_m \sum_{i=1}^n K\left ( {x-X_i \over h}\right ) \tau_{\alpha} (Y_i-m) - m_{\alpha,\widehat \theta(\alpha)}(x). $$
However, the latter estimator has the important drawback that the consideration of responses $Y_i$ in a neighborhood of $x$ induces a smoothing bias, whereas $\widehat r_{\alpha}(x)$ has no smoothing related bias, since it is based on the errors $Y_i-m_{\alpha,\theta(\alpha)}(X_i)$, whose conditional quantile of order $\alpha$ is exactly zero under $H_0$ for all $X_i$.

%The main advantage of our approach is that by considering conditional quantiles of the error $Y-m_{\alpha,\theta(\alpha)}(X)$ instead of conditional quantiles of the response $Y$, the smoothing method will not induce any bias.  Therefore, we will }

We suppose that $A$ is a closed subinterval of $(0,1)$. We define the following test statistic :
\begin{eqnarray} \label{TA}
&& \widehat T_A = \int_A \int_{R_X} \widehat r_{\alpha}(x)^2 w(x,\alpha) dx d\alpha,
\end{eqnarray}
for some weight function $w(x,\alpha)$.
For the case that $A$ contains only one value $\alpha$ we use
\begin{eqnarray} \label{TA2}
&&\widehat T_\alpha = \int _{R_X}\widehat r_{\alpha}(x)^2 w(x) dx \end{eqnarray}
for some weight function  $w(x)$. One could also generalize our results to the case that $A$ is a finite set. To keep notation simple we omit this case in our mathematical analysis.

Our test is an omnibus test that has power against all types of alternatives. It is based on the comparison of a kernel quantile estimator with the parametric fit. We will show that the test statistic is asymptotically equivalent to a weighted $L_2$-distance between the nonparametric and the parametric estimator. Similar tests have been used in a series of papers for mean regression. Early references are H\"ardle and Mammen (1993), Gonz\'alez-Manteiga and Cao-Abad (1993), Hjellvik, Yao and Tj\o stheim (1998), Zheng (1996) and Fan,  Zhang and Zhang (2001). Furthermore recent references are  Dette and Sprekelsen (2004),   Kreiss, Neumann and Yao (2008), Haag (2008), Leucht (2012), Gao and Hong (2008) and Ait-Sahalia, Fan and Peng (2009). Most of the more recent work concentrates on time series data.

The classical way to carry over results from parametric and nonparametric mean regression to quantile regression is the use of Bahadur expansions. The main point is that asymptotically quantile regression is equivalent to weighted mean regression. This approach has been used in Chaudhuri (1991), Truong (1989), He and Ng (1999), He, Ng and Portnoy (1998) and more recently in Hoderlein and Mammen (2009), Hong (2003), Kong, Linton and  Xia (2010), Lee and Lee (2008), El Ghouch and Van Keilegom (2009), Li and Racine (2008), and De Backer, El Ghouch and Van Keilegom (2017). A detailed review of quantile regression can be found in the book by Koenker (2005). Testing procedures in quantile regression were considered in Zheng (1998), Koenker and Machado (1999), Bierens and Ginther (2001), Horowitz and Spokoiny (2002), Koenker and Xiao (2002), and He and Zhu (2003), among others.  They all considered tests for the parametric form of the quantile function.
More recently, Rothe and Wied (2013) proposed a test statistic for the hypothesis that the conditional distribution belongs to a certain parametric class. 
Tests based on quantiles of the errors have also been considered in Su and White (2012) in the context of testing conditional independence.
Other recent papers are the ones by Volgushev et al.\ (2013) and Conde-Amboage, S\'anchez-Sellero and Gonz\'alez-Manteiga (2015), who considered significance tests in quantile regression and developed a test statistic based on marked empirical processes.

In this paper we will discuss how results from mean regression carry over to our case. Whereas elsewhere a first attempt could be based on the application of a Bahadur expansion, we will see that in our setting the accuracy of a direct application of Bahadur expansions is too poor. We will shortly explain this here for the testing problem where $A$ contains only one value $\alpha$.
Suppose for simplicity at this stage that the parametric model contains only one value $\theta_0=\theta_0(\alpha)$ and that
$\widehat \theta = \widehat \theta(\alpha) = \theta_0$. The Bahadur expansion of $\widehat r_{\alpha}(x)$ is given by
\begin{equation} \label{addhd} \widetilde r_{\alpha}(x)= -{\sumi K\Big(\frac{x-X_i}{h}\Big) \{I(\varepsilon_{i,\alpha}\leq 0) - \alpha\} \over \sumi K\Big(\frac{x-X_i}{h}\Big) f_{\varepsilon_\alpha|X}(0|X_i)}, \end{equation} where $f_{\varepsilon_\alpha|X}$ is the conditional density of $\varepsilon_\alpha$ given $X$.
This gives the following approximation for $\widehat T_\alpha$:
$$\widetilde T_\alpha = \int _{R_X}\widetilde r_{\alpha}(x)^2 w(x) dx.$$
One can show that up to a logarithmic factor $\sup_x |\widehat r_{\alpha}(x)-\widetilde r_{\alpha}(x)|$ and $\sup_x |\widehat r_{\alpha}(x)|$ are of order $(nh^d)^{-3/4}$ and $(nh^d)^{-1/2}$, respectively. This implies that up to a logarithmic factor, the difference  $\widehat T_\alpha - \widetilde T_\alpha $ is of order $(nh^d)^{-5/4}$. On the other hand as it is also the case in mean regression $\widetilde T_\alpha$ is equal to the sum of a deterministic term and a random term of order $n^{-1} h^{-d/2}$. Thus the above approximation only helps if $(nh^d)^{-5/4} << n^{-1} h^{-d/2}$ or equivalently if $nh^{3d} \to \infty$ for sample size $n$ going to $\infty$.  E.g.\ if one applies a bandwidth $h \sim n^{-1/(4+d)}$ that leads to rate optimal estimation of twice differentiable functions this assumption would allow only a one-dimensional setting $d=1$. Also in the case of minimax optimal testing with twice differentiable functions under the alternative (see Ingster (1993) and Guerre and Lavergne (2002)), the optimal bandwidth $h \sim n^{-2/(8+d)}$ is only allowed for dimension $d=1$. In this paper we develop an asymptotic theory for $L_2$-type quantile tests that works under the assumption that $nh^{3d/2} \to \infty$. In the above examples this allows dimensions $d\leq 7$ and $d\leq 3$. Furthermore, for our asymptotic discussion of the distribution of the test statistic on the hypothesis we only need the assumption that $nh^d \to \infty$. Thus on the hypothesis, our basic assumptions coincide with conditions needed for the asymptotics  of mean regression. We conjecture that also for the alternative the assumption $nh^{3d/2} \to \infty$ could be weakened but that then the asymptotic mean of the test statistic changes. We will comment on this after the statement of Theorem \ref{test}.

In our approach we will make use of the fact that Bahadur expansions of  kernel quantile estimators calculated at two different points are asymptotically independent if they are calculated at points that are such that the supports of the kernels do not overlap.  %that differ more than twice the bandwidth $h$. 
Thus the variance of an integral over a Bahadur expansion should be of smaller order than the variance of the Bahadur expansion at a fixed point. The main technical difficulty that will come up when applying this idea  is the need to calculate moments of the kernel regression quantiles. We will introduce a method for the expansion of such moments that is based on Edgeworth expansions in a related problem. Our main result gives a bound between the moments of kernel regression quantiles and the moments of its Bahadur approximation.

The paper is organized as follows. In the next section we will state our result on moments of kernel regression quantiles. Our main result on the asymptotics of $L_2$-type quantile tests is given in Section 3.  We will also introduce some kind of wild bootstrap procedure adapted to quantile regression and give a theoretical result on its consistency. %Our theory only applies for Nadaraya-Watson type smoothing. We add a result on tests based on local polynomial smoothing. This result is based on a direct application of Bahadur expansions and requires the stronger condition $nh^{3d} \to \infty$.
In Section 4 we present the results of a simulation study, and we analyze data on Engel curves. The proofs are postponed to the last three sections.

\section{Asymptotic moments}
In this section we will present an asymptotic result on higher order moments of kernel regression quantiles. This result will be our most important ingredient for getting our result on the asymptotic distribution of our test statistic. In our result the moments of kernel regression quantiles are compared with the moments of their Bahadur approximations. Recall that we are interested in the null hypothesis $H_0$ defined in (\ref{H0}).   We suppose that for all $\alpha \in A$,
\begin{eqnarray} \label{expm}
m_\alpha (\cdot) = m_{\alpha,\theta_0(\alpha)}(\cdot) + n^{-1/2} h^{-d/4} \Delta_\alpha(\cdot).
\end{eqnarray}
For the case $\Delta_\alpha\equiv 0$ the function $m_\alpha$ lies on the hypothesis. In order to develop our asymptotic theory, we need to work under the following assumptions. In the formulation of the assumptions and in the proofs we use the convention that $C, C_1, C_2, ... $ are generic strictly positive constants that are chosen large  enough, that $c, c_1, c_2, ... $ are generic strictly positive constants that are chosen small enough, and that $C^*, C^*_1, C^*_2, ... $ are generic strictly positive constants that are arbitrarily chosen. Using this convention we write $L_n= (\log n) ^C$ for a sequence with $C>0$ large enough and $L^*_n= (\log n) ^{C^*}$ for a sequence with an arbitrarily chosen constant $C^*>0$. All these variable names are used for different constants and sequences, even in the same equation.

We will make use of the following assumptions.

\begin{itemize}

\item[(B1)] The support $R_X$ of $X$ is a compact convex subset of $\IR^d$. The density $f_X$ of $X$ is bounded and bounded away from zero on $R_X$. The function $\Delta_\alpha$ is uniformly absolutely bounded for $\alpha \in A$.

\item[(B2)] The conditional distribution of $\varepsilon_{\alpha}$ given $X=x$ allows a density $f_{\varepsilon_{\alpha}|X}(e|x)$ that is twice differentiable with respect to $e$. For this derivative it holds that $|f^{\prime \prime} _{\varepsilon_{\alpha}|X}(e|x)| \leq C$ for $|e| \leq c$,  $x \in R_X$, and  $\alpha \in A$.  The density $f_{\varepsilon_{\alpha}|X}(e|x)$ also satisfies $f_{\varepsilon_{\alpha}|X}(e|x) > 0$ and 
$|f_{\varepsilon_{\alpha}|X}(e^{\prime}|x^{\prime})-f_{\varepsilon_{\alpha}|X}(e|x)|  \leq C (\|x^{\prime}-x\| + |e^{\prime}-e|)$ for $x,x^{\prime}\in R_X$ and $e,e^{\prime} \in \mathbb{R}$, where $\|\cdot\|$ is the Euclidean norm.  Moreover, the functions $f_X(x)$, $m_\alpha(x)$ and $\Delta_\alpha(x)$ are continuously differentiable with respect to $x$.  

\item[(B3)] The bandwidth $h$ satisfies $h=o(1)$ and $nh^{d}/L_n ^* \rightarrow \infty$. The kernel $k$ is a symmetric, continuously differentiable probability density function with compact support, $[-1,1]$, say. It fulfills a Lipschitz condition and it is monotone strictly increasing on $[-1,0]$. It holds that $k^\prime (k^{-1} (u)) \geq \min c \{u^{\kappa}, (k(0) - u)^{\kappa}\}$ for some $0 \leq \kappa < 1$ where $k^{-1}: [0, k(0)] \to [-1,0]$ denotes the inverse of $k:[-1,0] \to [0, k(0)]$.
\end{itemize}

%{\color{red} Here are the old versions of our assumptions. We need parts of it  that are not contained in the new (B1)-(B3) in (B4) ... for Theorem 2.
%
%\begin{itemize}
%
%\item[(B1)] The support $R_X$ of $X$ is a compact convex subset of $\IR^d$. The density $f_X$ of $X$ is strictly positive and the functions $f_X$ and $\Delta_\alpha$ are continuously differentiable on  $R_X$. The derivative of $\Delta_\alpha$ is uniformly absolutely bounded for $\alpha \in A$.
%
%
%\item[(B2)] The cumulative distribution function $F(\cdot|x)$ of the conditional distribution of $Y$ given $X=x$ is continuously differentiable with respect to $x$ and has a density $f(\cdot|x)$ that satisfies
%    \begin{eqnarray*}
%f(y|x) &>& 0,\\
%|f(y^{\prime}|x^{\prime})-f(y|x)|  &\leq&  C (\|x^{\prime}-x\| + |y^{\prime}-y|)\end{eqnarray*} for $x,x^{\prime}\in R_X$ and $y,y^{\prime} \in \mathbb{R}$, where $\|\cdot\|$ is the Euclidean norm.
%
%\item[(B3)] The bandwidth $h$ satisfies $h=o(1)$ and $nh^{3d/2} /L_n \rightarrow \infty$. The kernel $k$ is a symmetric, continuously differentiable probability density function with compact support, $[-1,1]$, say. It fulfills a Lipschitz condition and it is monotone strictly increasing on $[-1,0]$. It holds that $k^\prime (k^{-1} (u)) \geq \min c \{u^{\kappa}, (k(0) - u)^{\kappa}\}$ for some $0 \leq \kappa \leq 1$ where $k^{-1}: [0, k(0)] \to [-1,0]$ denotes the inverse of $k:[-1,0] \to [0, k(0)]$.
%\end{itemize}}

In our asymptotics, the density $f_X$ and the functions $\Delta_\alpha$ are fixed and do not depend on $n$. The cumulative distribution function $F(\cdot|x)$ of $Y$ given $X=x$ may depend on $n$. We do not indicate this in our notation.

Assumptions (B1)--(B3) are standard assumptions for the study of smoothing estimators, with the exception of the last assumption in (B3). We now shortly explain why this assumption is needed here.  For fixed $u$ and $x=(x_1,\ldots,x_d)^\intercal$, define the random vector $V_n= \sum_{j=1}^n \Big(k(\frac{x_1-X_{1,j}}{h}), ...,k(\frac{x_d-X_{d,j}}{h})\Big)^\intercal  \Big\{I(\varepsilon^\Delta_{j,\alpha} \le   \Delta_\alpha^h(x)  +u(nh^d)^{-1/2}) - \alpha\Big\}$ with $\varepsilon_{j,\alpha}^\Delta= \varepsilon_{j,\alpha} + n^{-1/2} h^{-d/4} \Delta_\alpha(X_j)$, and where $\Delta_\alpha^h(x)$ is defined in (\ref{Deltah}) below.  In the proof of the following Theorem  \ref{theolem4}
 we will develop Edgeworth expansions for the distribution of $V_n$. Typically, the summands of $V_n$ do not fullfil non-lattice type assumptions that are needed for the verification of Edgeworth expansions. But under (B3) a non-lattice assumption can be verified for the conditional distribution of a finite sum of summands of $V_n$. 
 %More precisely, for $p \geq 1$ we consider the conditional  density of $ \sum_{j=1}^p \Big(k(\frac{x_1-X_{1,j}}{h}), ...,k(\frac{x_d-X_{d,j}}{h})\Big) ^\intercal \Big\{I(\varepsilon^\Delta_{j,\alpha} \le   \Delta_\alpha^h(x)  +u(nh^d)^{-1/2}) - \alpha\Big\}$  given the value of $\varepsilon^\Delta_{j,\alpha}$ and given that $X_{\ell,j}$ lies in the support of $k(\frac{x_\ell-\cdot}{h})$ for $j=1,...,p$ and $\ell=1,\ldots,d$. From  the last assumption of (B3) we will get  that for $p=1$ this conditional density evaluated at the point $(v_1,\ldots,v_d)$ can be bounded by a constant times $(v_1 \cdot ... \cdot v_d)^{-\kappa}((k(0) - v_1) \cdot ... \cdot(k(0) - v_d))^{-\kappa}$. This fact will be used in the proof of Theorem \ref{theolem4}. In particular it will allow us to get bounds on the characteristic function of the density for $p>1$ and to show that the density fulfills a non-lattice condition if $p$ is chosen large enough. 
For more details we refer to the proof of Theorem \ref{theolem4}. The last assumption in (B3) can be easily verified. It just puts a simple bound on the derivative of $k^{-1}$. E.g., it can be easily checked for the triangle kernel and for all kernels of the form $k\left(z\right)=\mathbf{1}\left(\left|z\right|\leq1\right)c_r\left(1-z^2\right)^r$ with $r \geq 1$. In case that $k^\prime $ is bounded away from zero on bounded intervals of $(-1,0)$ the assumption follows if for some $l,l^*\in \mathbb N$, %$d>0$ 
 it holds that $k^\prime(x)= (x+1)^l + o((x+1)^l)$ for $x\geq -1$ and $x+1$ small enough and that $k^\prime(x)=  - x^{2l^*+1} + o(x^{2l^*})$ for $x$ in a neighborhood of $0$. %with some $l^* \geq 0$.

We put
\begin{eqnarray}  \label{rtildeadd}
 \widehat r^\Delta_{\alpha}(x) &= & \arg \min_r \sum_{i=1}^n K\left ( {x-X_i \over h}\right ) \tau_{\alpha} (\varepsilon^\Delta_{i.\alpha} -r),\\ \label{rtildeadd2}  \overline r_{\alpha}^{\Delta}(x) &= & \left\{
	\begin{array}{ll}
	\widehat r^\Delta_\alpha(x) & \mbox{ if } |\widehat r^\Delta_\alpha(x)| \le L_n (nh^d)^{-1/2} \\
	0 & \mbox{ otherwise}.
	\end{array}
	\right. \end{eqnarray}

	In the main result of this section we will consider conditional moments of the truncated kernel smoothing quantiles $ \overline r_{\alpha}^{\Delta}$, conditioned on the number of covariables falling into local neighborhoods. Note that, with positive probability, kernel smoothing quantiles are not defined because there is no covariable in the support of the kernel, with positive probability. Thus unconditional  moments are not defined. In the following theorem we will condition on local neighborhoods ${\cal N}^{-} (x) = \{ u: x_j-h \leq u_j \leq x_j +h \mbox{ for all } j=1,\ldots,d \}$ that are designed such that the result can be easily used for the asymptotic analysis of our test statistic in the next section. Note that $ {\cal N}^{-} (x)$ is the support of the kernel $h^{-d} K(h^{-1} [x- \cdot])$.  The theorem could also easily be stated with other local neighborhoods. The conditional moments of the truncated kernel smoothing quantiles $ \overline r_{\alpha}^{\Delta}$ will be compared with the conditional moments of the following modified Bahadur expansion, denoted by $ \widetilde r^\Delta_{\alpha}(x)$ :
	\begin{eqnarray} \label{rtilde} \widetilde r^\Delta_{\alpha}(x)&=& \widetilde r^{\Delta,-}_{\alpha}(x)
	+\Delta_\alpha^h(x), \end{eqnarray}
where $ \Delta_\alpha^h(x)$ is defined such that
\begin{eqnarray} \label{Deltah}
&& E^j_x \Big [K\Big(\frac{x-X_j}{h}\Big) \Big\{I(\varepsilon^\Delta_{j,\alpha} \le  \Delta_\alpha^h(x) ) - \alpha\Big\}\Big]=0.
\end{eqnarray}
Here $E_x^j$ denotes the conditional expectation, given that $j \in {\cal N^{-}}(x)$.	
Furthermore $\widetilde r^{\Delta,-}_{\alpha}(x)$ is defined as
	\begin{eqnarray*} \widetilde r^{\Delta,-}_{\alpha}(x)&=& -{\sumi K\Big(\frac{x-X_i}{h}\Big) \{I(\varepsilon^\Delta_{i,\alpha}\leq \Delta_\alpha^h(x)) - \alpha\} \over \sumi K\Big(\frac{x-X_i}{h}\Big) f_{\varepsilon_\alpha|X}( \Delta_\alpha^h(x,X_i) |X_i)}\\
	&=& -{\sumi K\Big(\frac{x-X_i}{h}\Big) \{I(\varepsilon_{i,\alpha}\leq \Delta_\alpha^h(x,X_i)) - \alpha\} \over \sumi K\Big(\frac{x-X_i}{h}\Big) f_{\varepsilon_\alpha|X}( \Delta_\alpha^h(x,X_i) |X_i)},  \end{eqnarray*}
	where $$
 \Delta_\alpha^h(x,X_j) =\Delta_\alpha^h(x) - n^{-1/2} h^{-d/4} \Delta_\alpha(X_j).$$
We have the following asymptotic result for the moments of kernel quantile estimators and their Bahadur approximations.
\medskip

	\begin{theorem} \label{theolem4}
Assume (B1)--(B3). Then, for natural numbers $l \geq 1$,
\begin{eqnarray}  \label{July3} E\Big\{\overline r_{\alpha}^{\Delta,-}(x)^{2l} - \widetilde r_{\alpha}^{\Delta,-}(x)^{2l} \Big| N^{-}(x)=m \Big\} = O(L_n (nh^d)^{-l-1}  ) ,\\
\label{July3add1} E\Big\{\overline r_{\alpha}^{\Delta,-}(x)^{2l-1} - \widetilde r_{\alpha}^{\Delta,-}(x)^{2l-1} \Big| N^{-}(x)=m \Big\} = O(L_n (nh^d)^{-l}  ) ,\end {eqnarray}
uniformly in $x \in R_X$, $\alpha \in A$ and $C_1^* nh^d\leq m \leq C_2^* nh^d$ where $N^{-}(x)$ is the random number of $X_i$'s that lie in ${\cal N}^{-} (x)$, and where
$ \overline r^{\Delta,-}_\alpha(x) =  \overline r^\Delta_\alpha(x) - \Delta_\alpha^h(x)$. For the second moments of  the uncentered estimators $\overline r_{\alpha}^{\Delta}$ and $ \widetilde r_{\alpha}^{\Delta}$  we have that
\begin{eqnarray}  \label{July3add} E\Big\{\overline r_{\alpha}^{\Delta}(x)^{2} - \widetilde r_{\alpha}^{\Delta}(x)^{2} \Big| N^{-}(x)=m \Big\} = O(L_n n^{-3/2}h^{-5d/4}   ).
\end {eqnarray}
Under the additional assumption that $\Delta_\alpha \equiv 0$ we get that 
\begin{eqnarray}  \label{July3addqq} E\Big\{\overline r_{\alpha}^{\Delta}(x)^{2} - \widetilde r_{\alpha}^{\Delta}(x)^{2} \Big| N^{-}(x)=m \Big\} = O(L_n (nh^d)^{-2}  ).
\end {eqnarray}
 \end{theorem}
	
We can apply the theorem when $\Delta_\alpha \equiv 0$, in which case $\Delta_\alpha^h \equiv 0$ and $\widetilde r_\alpha^\Delta(x) = \widetilde r_\alpha^{\Delta,-}(x)$ and $\overline r_\alpha^\Delta(x) = \overline r_\alpha^{\Delta,-}(x)$.  Hence, \eqref{July3} and \eqref{July3add1} hold with $\overline r_\alpha^{\Delta,-}(x)$ and $\widetilde r_\alpha^{\Delta,-}(x)$ replaced by $\overline r_\alpha^{\Delta}(x)$ and $\widetilde r_\alpha^{\Delta}(x)$.  In particular, for $l=1$ \eqref{July3addqq} follows directly from \eqref{July3}.

\section{Asymptotic theory}

 We suppose that there exists an estimator $\widehat \theta(\alpha)$ that converges to $\theta_0(\alpha)$. Hence, on the hypothesis the true value of $\theta(\alpha)$ is equal to $\theta_0(\alpha)$. On the alternative, $\theta_0(\alpha)$ may  depend on the chosen estimator $\widehat \theta(\alpha)$.

In order to develop the asymptotic distribution of $\widehat T_A$ and $\widehat T_\alpha$, we need the following additional assumptions.
\begin{itemize}
\item[(B4)] We assume that $$\sup_{ x \in R_X, \alpha \in A} | m_{\alpha,\widehat \theta(\alpha)}(x)  - m_{\alpha,\theta_0(\alpha)}(x)  - (\widehat \theta(\alpha) - \theta_0(\alpha)) ^\top \gamma_\alpha (x) | = O_P( n^{-{1\over 2}-c} )$$ for some function  $\gamma_\alpha (x)$. The  function $w(x)$ is continuous, and the functions $w(x,\alpha)$, $\gamma_\alpha (x)$ and $ \Delta_\alpha(x)$ are continuous with respect to $(\alpha,x)$. For $g(x) = w(x,\alpha)$ and $g(x) = w(x)$ it holds that $|g(x^{\prime})-g(x)|  \leq  C \|x^{\prime}-x\| $, and $|\gamma_\alpha(x^\prime)-\gamma_\alpha(x)| \le C \|x^{\prime}-x\|^{\delta} $ for some $0<\delta<1$ and for all $x,x^\prime \in R_X$ and all $\alpha \in A$.

\item[(B5)]  For some $\rho>0$ it holds that $$\sup_{\alpha \in A} \|\widehat \theta(\alpha) - \theta_0(\alpha)\| = O_P\left (  (n^{-1/2}h^{-d/4})/L_n^* \wedge (h^{-\delta} n^{-{1\over 2}-\rho} ) \wedge  (n^{-{1\over 4}}h^{d\over 4}/L_n^* )\right ).$$
\end{itemize}

The first assumption in (B4) can be shown under smoothness conditions on the relation $\theta \to m_{\alpha, \theta}(x) $. For the case that $A$ contains only one single element, this assumption in (B4) would  directly follow from (B5) and the assumption that $\theta \to m_{\alpha, \theta}(x) $ has a derivative that is continuous in  $x$. Assumption (B5) states that $\widehat \theta(\alpha)$ achieves at least a nearly parametric rate. In the case of linear quantile regression (i.e.\ $m_\alpha(X)=\theta(\alpha)^\top X$), such an assumption has been shown in Angrist et al (2006). Note that 
$\sup_{\alpha \in A} \|\widehat \theta(\alpha) - \theta_0(\alpha)\| = O_P\left ( n^{-1/2}\right )$ implies  (B5) if $\rho $ is chosen such that $h^{\delta} n^{\rho}\to 0$. Note that  $n^{-{1\over 2}} = o(n^{-{1\over 4}}h^{d\over 4}/L_n^*)$ because of 
$nh^d /L_n^*\to \infty$.

We now state our main result on the asymptotic distribution of our test statistics.

\begin{theorem} \label{test}
Assume (B1)-(B5). For the case that $\Delta_\alpha \not \equiv 0$ make the additional assumption that  $nh^{3d/2}/L_n ^* \rightarrow \infty$. Then,
\begin{eqnarray*} &&n h^{d/2} \widehat T_A - b_{h,A} \stackrel{d}{\rightarrow} N(D_A,V_A), \\
&&n h^{d/2} \widehat T_\alpha - b_{h,\alpha} \stackrel{d}{\rightarrow} N(D_\alpha,V_\alpha),
\end{eqnarray*}
where
\begin{eqnarray*}
D_A&=&\int_A \int_{R_X} \Delta_\alpha(x)^2 w(x,\alpha) \, dx \, d\alpha,\\
b_{h,A} &=& h^{-d/2} K^{(2)}(0)  \int_A \alpha(1-\alpha) \int_{R_X} \frac{w(x,\alpha)}{f_X(x) f^2_{\varepsilon_\alpha|X}(0|x)} dx \, d\alpha, \\
V_A &=& 4 K^{(4)}(0)  \int_{\alpha, \beta \in A,  \alpha < \beta} \alpha^2 (1-\beta)^2  \int _{R_X} \frac{w(x,\alpha) w(x,\beta)}{f_X^2(x) f_{\varepsilon_\alpha|X}^4(0|x)} dx \, d\alpha\ d\beta,
\end{eqnarray*}
\begin{eqnarray*}
D_\alpha&=& \int_{R_X} \Delta_\alpha(x)^2 w(x) \, dx ,\\
b_{h,\alpha} &=& h^{-d/2} K^{(2)}(0)   \alpha(1-\alpha) \int_{R_X} \frac{w(x)}{f_X(x) f^2_{\varepsilon_\alpha|X}(0|x)} dx, \\
V_{\alpha} &=& 4 K^{(4)}(0)  \alpha^2 (1-\alpha)^2  \int _{R_X} \frac{w^2(x)}{f_X^2(x) f_{\varepsilon_\alpha|X}^4(0|x)} dx,
\end{eqnarray*}
and where for any $j$, $K^{(j)}(0)$ denotes the $j$-times convolution product of $K$ at $0$.
\end{theorem}

In our theorem for the alternative we make the additional assumption that $nh^{3d/2}/L^*_n$ converges to $\infty$. This assumption is used in  the proof for the treatment of the deterministic  term $T_{n,2}$, see Lemma  \ref{lem4}. The assumption $nh^{3d/2}/L^*_n\to \infty$ can be weakened but with another limit for $T_{n,2}$. This would result in a limit theorem for the test statistic with a mean that differs from $b_{h,\alpha}$. We have added a short discussion of this point after the statement of Lemma  \ref{lem4}.

We expect that Theorem \ref{test} cannot be used for an accurate calculation of critical values.  The asymptotic normality result of
Theorem \ref{test} is based on the fact that kernel smoothers are asymptotically independent if they are calculated at points that differ more than $2h$. Thus the convergence is comparable to the convergence of the sum of $h^{-d}$ independent summands. This would motivate a rate of convergence of order $h^{-d/2}$. As has been suggested for  other goodness-of-fit tests in the literature, also here a way out is to use a bootstrap procedure. We will introduce some kind of wild  bootstrap for quantiles in which the Bahadur expansion $\widetilde r_{\alpha}$ of $\widehat r_{\alpha}$ is resampled. For the definition of $\widetilde r_{\alpha}$ see (\ref{addhd}) in Section 2. For the bootstrap, we define
 $$\widetilde r^*_{\alpha}(x)= -{\sumi K\Big(\frac{x-X_i}{h}\Big) \{I(U_{i}\leq \alpha) - \alpha\} \over \sumi K\Big(\frac{x-X_i}{h}\Big) \widehat f_{\varepsilon_\alpha|X}(0|X_i)},$$
where $\widehat f_{\varepsilon_\alpha|X}$ is an estimator of $f_{\varepsilon_\alpha|X}$ and $U_i$ are independent random variables with uniform distribution on $[0,1]$ that are independent of the sample. The bootstrap test statistics are defined as:
$$ \widehat T^{*}_A = \int_A \int_{R_X} \widetilde r_{\alpha}^{*}(x)^2 w(x,\alpha) dx\  d\alpha $$
and
$$ \widehat T^{*}_\alpha = \int _{R_X}\widetilde r_{\alpha}^{*}(x)^2 w(x) dx. $$
For proving the consistency of this bootstrap procedure, we do not specify the choice of the estimator $ \widehat f_{\varepsilon_\alpha|X}$ that is used in the construction of the bootstrap procedure. We only assume that the estimator is consistent:

\begin{itemize}
\item[(B6)] It holds that
$$\sup _{\alpha\in A,\ x \in R_X} \left | \widehat f_{\varepsilon_\alpha|X}(0|x) -f_{\varepsilon_\alpha|X}(0|x)\right | \rightarrow 0, $$
in probability.
\end{itemize}

The next theorem shows the consistency of the above bootstrap approach.

\begin{theorem} \label{testboot1}
Assume (B1)-(B6). Then,
\begin{eqnarray*} && d_K({\cal L}^*(n h^{d/2} \widehat T^*_{A} - b_{h,A}) , N(D_A,V_{A})) \stackrel{p}{\rightarrow} 0, \\
&& d_K({\cal L}^*(n h^{d/2} \widehat T^*_{\alpha} - b_{h,\alpha}) , N(D_\alpha,V_{\alpha}))\stackrel{p}{\rightarrow} 0,
\end{eqnarray*}
where ${\cal L}^*(...)$ denotes the conditional distribution, given the sample. Furthermore, $d_K$ is the Kolmogorov distance, i.e.\ the sup norm of the difference between the corresponding distribution functions.
\end{theorem}

Theorem \ref{testboot1} remains to hold if we replace (B1)--(B5) by weaker conditions. We do not pursuit this because we need for consistency of bootstrap that both, Theorem \ref{test} and Theorem \ref{testboot1}, hold.

\section{Numerical study}
In this section, we present the results of our numerical studies. In our first simulation we show that a direct application of the Bahadur representation is not accurate enough for studying the approximation of the distribution of our test statistics $\widehat T_{\alpha}$ and $\widehat T_A$. For this purpose, we compare the differences $D_1=\int |\widehat r_{\alpha}(x)^2 - \widetilde r_{\alpha}(x)^2|w(x)dx $ and $D_2=|\int \widehat r_{\alpha}(x)^2w(x)dx - \int \widetilde r_{\alpha}(x)^2w(x)dx| $.  Here $D_1$ is the integrated difference between the quantile regression and its Bahadur representation and $D_2$ is the difference between the test statistic and its approximation based on Bahadur representation. It is clear that $D_2 \leq D_1$. Our point is not that $D_2$ is smaller than $D_1$ but that the ratio $D_1/D_2$ is large and that it is decreasing for an increasing bandwidth.  This result supports our theory that a direct use of Bahadur expansions only works under very restrictive assumptions on the bandwidth. Table \ref{tab:diff1d} shows the results of $D_1$ and $D_2$ for the one dimensional case. We also simulated a two dimensional model, whose results are shown in Table \ref{tab:diff2d}. In the one dimensional model, we set $Y_i=X_i + (0.5X_i+0.5)\epsilon_i$, where $\epsilon_i$ has a standard normal distribution. This results in the  $\alpha^{th}$ quantile function $m_{\alpha}(x) = x+0.5 z_{\alpha} (x+1)$, where $z_{\alpha}$ stands for the  $\alpha^{th}$ quantile of the standard normal distribution. For the two dimensional model, we set $Y_i=X_{1i}+X_{2i} + (0.5X_{1i}-0.5X_{2i}+1)\epsilon_i$, where $\epsilon_i$ has a standard normal distribution and we get the $\alpha^{th}$ quantile function $m_{\alpha}(x_1,x_2) = x_1+2x_2+0.5 z_{\alpha} (x_1-x_2+2)$. For the one dimensional model, we generated $X_i$ from the uniform distribution supported on the unit interval $(0,1)$. For the two dimensional model we generated $(X_{1i},X_{2i})$ from a distribution on the unit square $(0,1)\times (0,1)$ which has uniform marginals  but where the joint distribution differs from a uniform distribution. This is done to allow for a dependence between the two regressors. We  generated random vectors from a bivariate normal distribution with correlations $\rho=0.2$ and 0.8 and then transformed them with their marginal distribution functions. We generated 400 data sets of size 200 and 400 for each model. For the one dimensional model, we used the bandwidths $h=0.05, 0.08, 0.1, 0.12 \mbox{ and } 0.15$ and we used the bandwidths $h= 0.125, 0.150, 0.175 \mbox{ and } 0.200$ for the two dimensional model.  We used the R package {\sl quantreg} for fitting quantile functions. From Table \ref{tab:diff1d} and Table \ref{tab:diff2d}, one can see that the ratio of $D_1$ over $D_2$ is large for small bandwidths and decreases as the bandwidth grows. This observation supports our approach for the asymptotic theory. This implies that when we approximate the test statistic it requires less strict assumptions on the bandwidth if we approximate the integrated function rather than when we approximate the quantile function itself.

\begin{table}[h!]
\begin{center}
\begin{tabular}{c r c| r r r r r }
\hline \multicolumn{3}{c|} {}&\multicolumn{5}{c}{Bandwidth}\\ \hline
          &                       &           & 0.05  &   0.08  & 0.1     & 0.12    & 0.15   \\ \hline
$n=200$  & $\alpha=0.25$ &$D_1$  &0.195&  0.116 &  0.091& 0.077& 0.059 \\
          &                       &$D_2$  & 0.067& 0.044&   0.037& 0.034&  0.028  \\
          &                       &Ratio    & \it 2.922& \it 2.653& \it 2.473& \it 2.284& \it 2.073 \\
          & $0.5$\phantom{0}&$D_1$  & 0.231 &  0.139& 0.110& 0.091& 0.071  \\
          &                            &$D_2$ &0.124&0.078& 0.065& 0.056& 0.046  \\
          &                       &Ratio    &\it 1.865& \it 1.779& \it 1.687& \it 1.635&\it 1.547   \\
          & $0.75$ &$D_1$ &   0.196& 0.117& 0.092& 0.077& 0.059\\
          &                       &$D_2$  & 0.065& 0.044& 0.037& 0.033& 0.029   \\
          &                       &Ratio    & \it  3.028& \it 2.696& \it 2.481& \it 2.305&\it  2.070    \\ \hline
$n=400$  & $\alpha=0.25$ &$D_1$ & 0.094& 0.057&  0.045& 0.037& 0.029  \\
          &                       &$D_2$ & 0.033& 0.022& 0.019& 0.017& 0.014   \\
          &                       &Ratio    & \it 2.899& \it 2.536& \it 2.338& \it 2.170& \it 2.027   \\
          & $0.5$\phantom{0}&$D_1$& 0.109& 0.068& 0.054& 0.045& 0.036  \\
          &                       &$D_2$  & 0.058& 0.039& 0.032& 0.028& 0.023 \\
          &                       &Ratio    & \it  1.867& \it 1.755& \it 1.684& \it 1.614& \it 1.552  \\
          & $0.75$ &$D_1$ &    0.094& 0.056& 0.045& 0.037& 0.029 \\
          &                       &$D_2$  &  0.032& 0.022& 0.019& 0.017& 0.015 \\
          &                       &Ratio    &\it 2.959&\it 2.604&\it  2.378& \it 2.192&\it 1.982\\ \hline

\end{tabular}
\end{center}
\caption{Difference between two approximations: $D_1$ is the integrated squared approximation error of a quantile estimator by its Bahadur representation and $D_2$ is the approximation error of the test statistic $\widehat T_{\alpha}$ by $\widetilde T_{\alpha}$. }
\label{tab:diff1d}
\end{table}

\begin{sidewaystable}
%\begin{table}[f]
\begin{center}
\begin{tabular}{c r c| r r r r | r r r r }
\multicolumn{3}{c|} {}&\multicolumn{4}{c}{Weak dependence} &\multicolumn{4}{c}{Strong dependence} \\
\hline \multicolumn{3}{c|} {}&\multicolumn{8}{c}{Bandwidths}\\  \hline
        &              &      &   0.125 & 0.150& 0.175 & 0.200     &   0.125 & 0.150& 0.175 & 0.200   \\ \hline
$n=200$ & $\alpha=0.25$ & $D_1$ &0.249&0.139&0.086&0.058 &0.153&0.105&0.077&0.059  \\
         &                       & $D_2$ &0.081&0.051&0.036&0.029 &0.044&0.037&0.033&0.029  \\
         &                       & Ratio & \it 3.065&\it 2.716&\it 2.369&\it 1.983 &\it 3.445&\it 2.828&\it 2.371&\it 2.024  \\ \hline
         & $            0.5 $ & $D_1$&0.296&0.164&0.102&0.072 &0.180&0.126&0.094&0.072 \\
         &                       & $D_2$ &0.142&0.087&0.061&0.048 &0.089&0.069&0.056&0.048 \\
         &                       & Ratio &\it 2.093&\it 1.884&\it 1.667&\it 1.500 &\it 2.021&\it 1.832&\it 1.659&\it 1.504 \\ \hline
         & $           0.75$ & $D_1$&0.247&0.138&0.085&0.058 &0.153&0.106&0.077&0.059 \\
         &                       & $D_2$ &0.077&0.047&0.035&0.029 &0.045&0.037&0.032&0.029 \\
         &                       &Ratio    &\it 3.221&\it 2.929&\it 2.446&\it 1.990 &\it 3.355&\it 2.831&\it 2.378&\it 2.022 \\ \hline
$n=400$ &$\alpha=0.25$ &$D_1$&0.249&0.139&0.086&0.058 &0.074&0.051&0.038&0.029 \\
                      &&$D_2$& 0.081&0.051&0.036&0.029 &0.023&0.019&0.017&0.015 \\
                       &&Ratio&\it  3.066&\it 2.716&\it 2.369&\it 1.983 &\it 3.273&\it 2.716&\it 2.278&\it 1.958 \\ \hline
              &$0.5 $ &$D_1$& 0.296&0.166&0.102&0.072 &0.087&0.061&0.046&0.035 \\
                     &&$D_2$&  0.142&0.087&0.061&0.048 &0.044&0.034&0.028&0.024 \\
                      &&Ratio& \it  2.093&\it 1.884&\it 1.667&\it 1.500 &\it 1.991&\it 1.785&\it 1.610&\it 1.467 \\ \hline
            &$ 0.75$ &$D_1$&  0.247&0.138&0.085&0.058 &0.073&0.051&0.037&0.028 \\
                     &&$D_2$&  0.077&0.047&0.035&0.029 &0.022&0.019&0.017&0.015 \\
                      &&Ratio& \it  3.221&\it 2.929&\it 2.446&\it 1.990 &\it 3.274&\it 2.666&\it 2.230&\it 1.899 \\ \hline
\end{tabular}
\end{center}
\caption{Difference between two approximations under a two dimensional model: $D_1$ is the integrated squared approximation error of a quantile estimator by its Bahadur representation and $D_2$ is the approximation error of the test statistic $\widehat T_{\alpha}$ by $\widetilde T_{\alpha}$. The left panel in the table is the result for $\rho=0.2$ and the right panel shows the result for $\rho=0.8$.}
\label{tab:diff2d}
%\end{table}
\end{sidewaystable}

The second simulation study is conducted to show the validity of our bootstrap procedure. We considered four scenarios:
\begin{enumerate}
\item[I.] All quantiles are linear:
$$ Y_i = m_0(X_i) +\sigma_0(X_i)\epsilon_i.$$
\item[II.] The median is linear and other quantiles are not linear:
$$ Y_i = m_0(X_i) +\sigma_1(X_i)\epsilon_i.$$
\item[III.] All quantiles are non-linear:
\begin{eqnarray*}
(a) & & Y_i = m_1(X_i) +\sigma_0(X_i)\epsilon_i;  \\
(b) & & Y_i = m_1(X_i) +\sigma_1(X_i)\epsilon_i.
\end{eqnarray*}
\end{enumerate}

Here $m_0(x)=x$, $m_1(x)=\sin(2\pi(x-0.5))$, $\sigma_0(x)=\frac 1 2 (1+x)$, and $\sigma_1(x)=2 (1.1+\sin(2\pi(x-0.5)))$. The covariates $X_i$ are generated from a uniform distribution on the unit interval $(0,1)$. We generated 200 samples of size 400. We generated 201 bootstrap samples for each data set. The three scenarios are shown in Figure \ref{fig:boots_models}. In the bootstrap procedure, we used a kernel density estimator for estimating the conditional density $f_{\varepsilon_\alpha|X}(0|x)$.

We tried three bandwidths 0.075, 0.100, and 0.125 for the test statistic, and fifteen choices of bandwidths $(0.1,0.2,0.3)\times(0.050,0.075,0.100,0.125,0.150)$ for estimating the conditional density of the error used in the bootstrap procedure. We applied the proposed bootstrap test for testing the linearity of the lower quartile, the median, and the upper quartile functions. We also tested the linearity hypothesis over these three different quantile levels. In our scenarios there are four models under the null hypothesis: all three quantiles in scenario I and the median in scenario II.  In Table \ref{tab:boots}, we report the summary statistics of rejection ratios of 45 different choices of bandwidths. Among 45 different choices of bandwidths, there was no case where the bootstrap test did not keep the significance level of 5\% under scenario I. In slightly more than half of the cases  the bootstrap test did not keep the significance level in testing the linearity of the median under scenario II. These cases appeared when we used large bandwidths. %However, 10 cases among those cases showed a rejection rate below 7\% and only in 5 cases the bootstrap test rejected more than 15 \% of the time. 
Concerning the power of the bootstrap test, we observed that almost all choices of the bandwidth showed a power near one under scenario III-(a). One exception is the case with the smallest bandwidths where we observe an empirical power around 0.8. One interesting observation is that the bootstrap test shows a poor power for the lower quartile in scenario III-(b) where the empirical power ranges from 0.045 to 0.27. This is however natural since the function is not so far from a linear function as one can see in Figure \ref{fig:boots_models}. We also observed that the median and the upper quartile in scenario III-(b) showed much stronger power. The result of the test based on the test statistic integrated over levels shows a similar result. In this case, only scenario I is in the null hypothesis and there was no case where the empirical size of the bootstrap test is bigger than 5\%. The power behavior is also similar. The test showed the strongest power with the bigger bandwidths.

\begin{figure}
\begin{center}
\scalebox{1.2}{\includegraphics{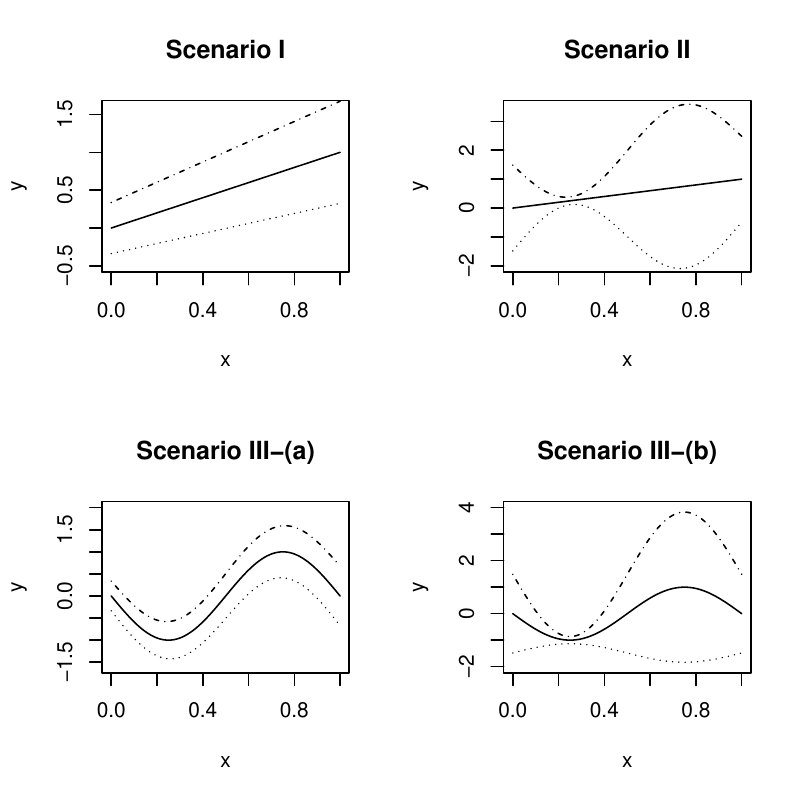}}
\end{center}
\vspace*{0cm}
\caption{Shape of quantile curves in each scenario. The three curves in each panel represent the $0.25$, $0.5$, and $0.75-$quantile curves in each scenario.}\label{fig:boots_models}
\end{figure}

\begin{table}[h!]
\begin{center}
\begin{tabular}{l  l| r r r r  }
\hline \multicolumn{2}{c|} {}&\multicolumn{4}{c}{Quartile functions}\\
         &       & 1st Quartile & Median & 3rd Quartile & Sum quartiles \\ \hline
Scenario I %& Min   &0.000         &0.000   & 0.000        &0.000  \\
           & Q1    &0.005         &0.000   & 0.005        &0.000  \\
           & Med   &0.010         &0.005   &  0.010       &0.005  \\
%           & Mean  &0.009         &0.006   & 0.011        &0.006  \\
           & Q3    &0.010         &0.010   & 0.015        &0.010      \\
%           & Max   &0.035 &0.025   &0.035   &0.015   \\ \hline
Scenario II %& Min   &0.220   &0.005 &0.310 &0.145 \\
            & Q1     &0.470   &0.030 &0.450 &0.380 \\
            & Med  &0.575   &0.055 &0.545 &0.550 \\
%            &  Mean &0.548   &0.073 &0.534 &0.526 \\
            & Q3 &0.630   &0.105 &0.610 &0.695 \\
%           &  Max &0.740   &0.260 &0.770 &0.840 \\ \hline
Scenario III (a) %& Min   &0.825   &0.990 &0.800 &0.940 \\
                 & Q1 &0.985   &1.000 &0.995 &1.000 \\
                 & Med &1.000   &1.000 &1.000 &1.000 \\
%                 & Mean &0.985   &1.000 &0.988 &0.998 \\
                 & Q3 &1.000   &1.000 &1.000 &1.000 \\
%                 & Max  &1.000   &1.000 &1.000 &1.000 \\ \hline
Scenario III (b) %& Min    &0.045   &0.205 &0.150 &0.070 \\
                & Q1    &0.065   &0.300 &0.735 &0.545 \\
                & Med &0.135   &0.385 &0.865 &0.770 \\
%                & Mean &0.124   &0.383 &0.794 &0.687 \\
                 & Q3  &0.155   &0.440 &0.935 &0.895 \\
%                 & Max  &0.270   &0.605 &0.990 &0.990 \\ \hline
\end{tabular}
\end{center}
\caption{First quartile, median and third quartile (obtained from 45 choices of the bandwidth) of the rejection proportions based on 200 generated samples.}
\label{tab:boots}
\end{table}

Figure \ref{fig:pval} shows the distributions of estimated p-values by using the proposed bootstrap. The plots are based on 200 simulated data sets for scenario II. The left panel shows the distribution of estimated p-values for testing the linearity of the median, which lies in the null and the right panel shows the distribution of estimated p-values for testing the linearity of the upper quartile, which lies in the alternative. The distribution in the left panel is close to the uniform distribution which we expect for the null hypothesis and the right panel shows that the bootstrap p-values are close to zero, which we also expect.

To summarize our observations from this simulation study, in our setting the bootstrap test keeps the level well except for cases where we use too large bandwidths. On the other hand, too small bandwidths lead to relatively poor power. Interesting cases are scenario II and scenario III-(b). Under scenario II, the median is linear but both quartiles are not. The simulation result shows that the bootstrap test keeps the level for the median and has some power for the other quantiles. Under scenario III-(b), the lower quartile is non-linear but close to the null. We observe that here the bootstrap test has stronger power for the median and the upper quartile than for the lower quartile.

\begin{figure}
\begin{center}
\vspace*{0cm}
\scalebox{0.8}{\includegraphics{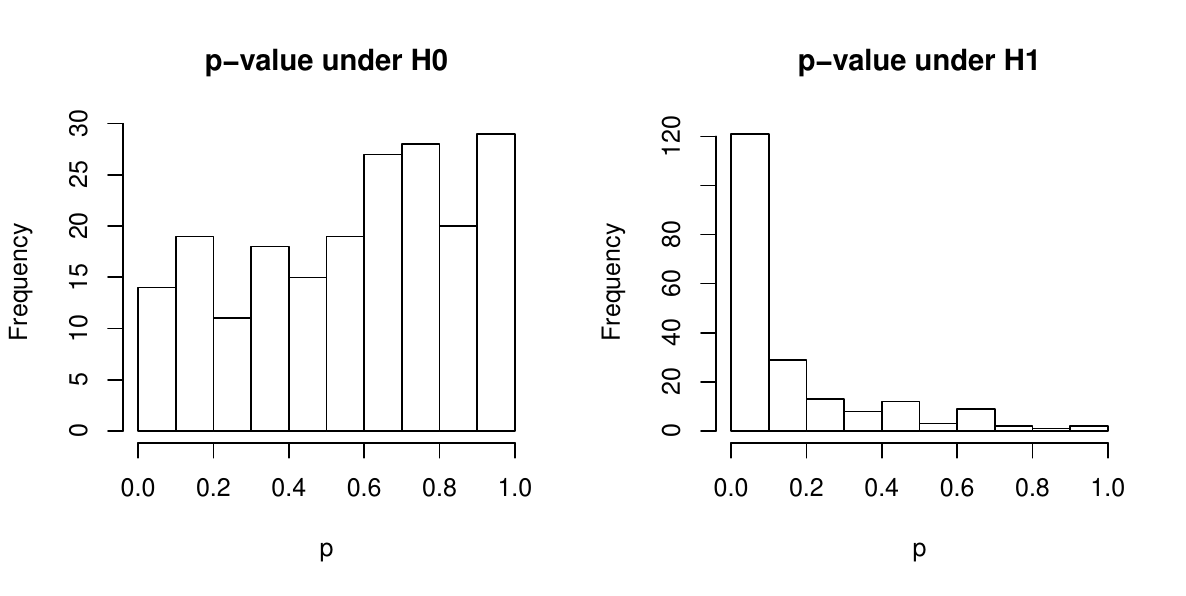}}
\end{center}
\vspace*{0cm}
\caption{The left panel shows the distribution of estimated p-values under the null  and the right panel shows the distribution of estimated p-values under the alternative.}
\label{fig:pval}
\end{figure}

%\begin{sidewaysfigure}
%\begin{center}
%\vspace*{-4cm}
%\includegraphics[width=15cm]{p-value.pdf}
%\end{center}
%\vspace*{-7cm}
%\caption{The left panel shows the distribution of estimated p-values under the null  and the right panel shows the distribution of estimated p-values under the alternative.}
%\label{fig:pval}
%\end{sidewaysfigure}

%{\color{red} HERE I ADDED A NEW SIMULATION: COMPARISON WITH ZHENG (1988). I THINK THIS COULD BE ONLY IN THE LETTERS TO REVIEWERS.}
In the last simulation study, we compared our test with the test proposed by Zheng (1998). We considered the same four scenarios as in the previous simulation study. %for checking the validity of the bootstrap approximation for finite samples.
In this simulation, we generated 500 data sets of 400 observations. For the bootstrap we generated 501 bootstrap samples. The other simulation settings are the same as in the previous simulation study. To choose the bandwidth for Zheng's test, we applied the function {\sl npregbw} in the {\sl R}-package {\sl np}, which is based on cross-validation with AIC. For the nonparametric quantile estimator in our procedure, we used the bandwidth proposed in Yu and Jones (1998). The bandwidths for the kernel estimator for the conditional density in the bootstrap procedure were chosen by a rule of thumb using the function {\sl bw.nrd} in {\sl R}. The level of the tests was set to $0.05$. We observe in Table \ref{tab:Zheng} that the level of our test is close to the nominal level. None of the two tests is always more powerful than the other. In the null model both tests keep the level well. In Scenario II, Zheng's test shows stronger power than the proposed test, whereas in Scenario III(b), the proposed test has higher power.

\begin{table}[h!]
\begin{center}
\begin{tabular}{l  l| r r r r  }
\hline \multicolumn{2}{c|} {}&\multicolumn{4}{c}{Quartile functions}\\
           &       & 1st Quartile &  Median& 3rd Quartile & Sum quartiles \\ \hline
Scenario I & Zheng   &0.006      &0.000   &0.000     & $\cdot$  \\
           & MVY   &0.024    &0.052    &0.048    &0.028  \\ \hline
Scenario II & Zheng   &0.954     &0.008      &0.950      &$\cdot$   \\
           & MVY   &0.696   &0.002     &0.650     &0.884  \\ \hline
Scenario III (a) & Zheng   &0.974      &0.994      &0.984      &$\cdot$   \\
           & MVY   &0.994     &0.992    &0.990     &1.000  \\ \hline
Scenario III (b) & Zheng   &0.034      &0.334    &0.460     &$\cdot$   \\
           & MVY   &0.048    &0.404     &0.992    &0.954  \\ \hline
\end{tabular}
\end{center}
\caption{Rejection proportions based on 500 generated samples.}
\label{tab:Zheng}
\end{table}

Finally, as an illustrating example, we applied the proposed test to a historic data set of Ernst Engel. The data set was used in  Koenker (2005), among many other publications. The data set was first presented by Engel (1857) to support his famous Engel's law. The data set has two variables, household income and food expenditure and it contains 235 observations. Figure \ref{fig:engel_data} shows the scatter plot of this dataset and the scatter plot of the data after a log transform with base 10. As one can see in  Figure \ref{fig:engel_data}, there is one outlier.  We removed this point from the data. Hence the analysis below is based on 234 observations.  We first analyzed the log transformed income versus the log transformed food expenditure. We used five different bandwidths for calculating the test statistic. The bandwidth for the conditional density estimator used in the bootstrap resampling was chosen by a rule of thumb. To obtain the bootstrap distribution, we resampled  the data set 1,001 times. We test the linearity of quantiles for $\alpha = $ 0.1,     0.2,      0.3,       0.4,    0.5,     0.6,     0.7,   0.8,    and 0.9.  As can be seen from Table \ref{tab:engel}, the test did not reject linearity of quantiles for any of these values at the significance level 5\%. This was the case for all five bandwidth choices.
We also used the log transformed income with the original untransformed food expenditure as a further example. Figure \ref{fig:engel_datanon} shows the scatter plot of this dataset  together with 0.1, 0.3, 0.5, 0.7, 0.9 linear quantile fits. The figure shows that high level quantiles deviate from their linear fits. This is also seen by our test, since it rejects the linearity for high level quantiles. The estimated p-values for the conditional quantiles of level $0.5$ or higher are smaller than $0.05$ for every bandwidth we used. % except one case: the $0.55-$quantile with bandwidth $0.075$.

%\begin{sidewaystable}
\begin{table}
\begin{center}
\begin{tabular}{l r r r r r r r r r r r r r r r r r }
\multicolumn{10}{c} {$\log_{10}( \mbox{income}) ~vs~ \log_{10} (\mbox{food expenditure})$ }\\
\hline \multicolumn{1}{c} {}&\multicolumn{9}{l}{Quantile level}\\
        &0.1&0.2&0.3&0.4&0.5&0.6&0.7&0.8&0.9 \\
Bandwidth & & & & & & & & & \\                         \hline
0.050&0.289&0.246&0.629&0.999&0.996&1.000&0.999&1.000&1.000 \\
0.075&0.152&0.383&0.997&0.992&0.993&0.979&1.000&1.000&1.000 \\
0.100&0.105&0.745&0.997&0.964&0.971&0.970&0.999&1.000&0.993 \\
0.125&0.100&0.986&0.988&0.908&0.895&0.992&1.000&1.000&0.963 \\
0.150&0.149&0.996&0.976&0.894&0.843&0.997&0.999&1.000&0.935  \\ \hline \\
\multicolumn{10}{c} {$\log_{10}( \mbox{income}) ~vs~ \mbox{food expenditure}$ }\\  \hline
 \multicolumn{1}{c} {}&\multicolumn{9}{l}{Quantile level}\\

&0.1&0.2&0.3&0.4&0.5&0.6&0.7&0.8&0.9 \\
Bandwidth & & & & & & & & & \\                         \hline
0.050&0.305&0.675&0.413&0.135&0.031&0.008&0.012&0.004&0.000 \\
0.075&0.569&0.498&0.268&0.140&0.046&0.027&0.003&0.002&0.000 \\
0.100&0.640&0.324&0.181&0.047&0.015&0.004&0.003&0.001&0.000 \\
0.125&0.531&0.235&0.150&0.034&0.009&0.001&0.002&0.001&0.000  \\
0.150&0.664&0.197&0.088&0.012&0.003&0.002&0.002&0.003&0.000  \\ \hline
\end{tabular}
\caption{Estimated p-values for testing the linearity of conditional quantiles of Engel's data. The upper table shows the estimated p-values for testing the linearity of conditional quantiles of $\log_{10} (\mbox{food expenditure})$ as a function of $\log_{10}( \mbox{income})$ and the lower table shows the estimated p-values for testing the linearity of conditional quantiles of $\mbox{food expenditure}$ as a function of $\log_{10}( \mbox{income})$.}
\label{tab:engel}
\end{center}
\end{table}
%\end{sidewaystable}

\begin{figure}[!h]
\begin{center}
\vspace*{-1cm}
\scalebox{0.9}{\includegraphics{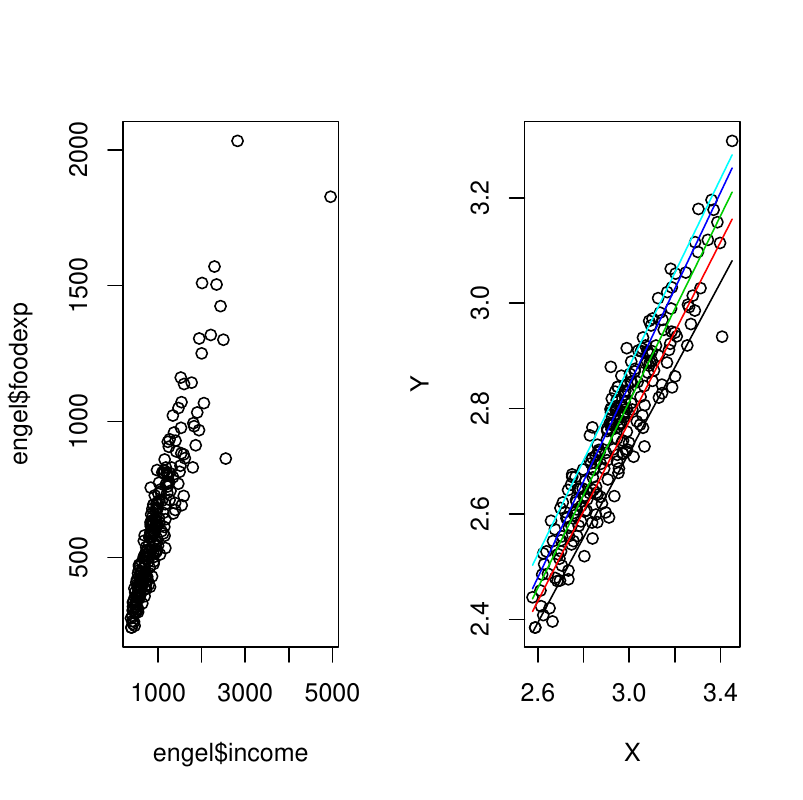}}
\vspace*{-1cm}
\end{center}
\caption{The left panel shows the scatter plot of the original Engel data. The right panel shows the scatter plot of log transformed data after removing one influential point. The lines in the right panel represent linear quantile fits of levels 0.1, 0.3, 0.5,   0.7, and 0.9.}
\label{fig:engel_data}
\end{figure}

\begin{figure}[!h]
\begin{center}
\vspace*{0cm}
\scalebox{0.8}{\includegraphics{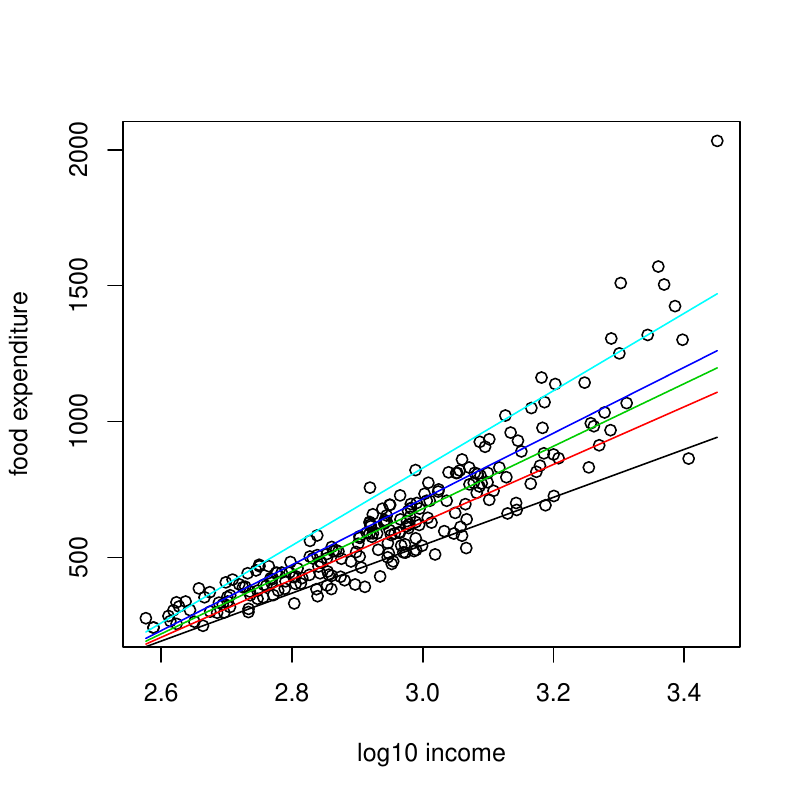}}
\end{center}
\vspace*{0cm}
\caption{The figure shows the scatter plot of log transformed income versus food expenditure after removing one influential point. The lines represent linear quantile fits of level 0.1, 0.3, 0.5, 0.7, and 0.9.}
\label{fig:engel_datanon}
\end{figure}

\section{Proof of Theorem \ref{theolem4}}

We need to show equations \eqref{July3}--\eqref{July3add}. Claim \eqref{July3add} follows from
\eqref{July3}--\eqref{July3add1} because of $\Delta_\alpha^h(x) = O(n^{-1/2}h^{-d/4})$. Furthermore,  \eqref{July3addqq} is a direct consequence of  \eqref{July3}, see the remark after the statement of the theorem. It remains to show \eqref{July3}--\eqref{July3add1}.

It holds that $P(cm_0 \leq N^-(x) \leq C m_0) \to 1$, where we use the shorthand notation $m_0=nh^d$.  At this point  and in the following proofs we will make use of our convention of using the symbols, $c,C,...$.

 First note that
\begin{eqnarray} \label{start} && \widehat r^{\Delta,-}_\alpha(x) \le um_0^{-1/2} \:\: \mbox{  if and only if } \\ \nonumber && \qquad \qquad \sum_{j\in {\cal N^{-}}(x)} K\Big(\frac{x-X_j}{h}\Big) \Big\{I(\varepsilon^\Delta_{j,\alpha} \le  \Delta_\alpha^h(x) + um_0^{-1/2}) - \alpha\Big\} \ge 0. \end{eqnarray}

Let
\begin{eqnarray*}
g_{x,\alpha}(u) &=& E^j_x \Big[K\Big(\frac{x-X_j}{h}\Big) \Big\{I(\varepsilon^\Delta_{j,\alpha} \le  \Delta_\alpha^h(x) + um_0^{-1/2}) - \alpha \Big\}  \Big] \\
&=& E^j_x \Big[K\Big(\frac{x-X_j}{h}\Big) \Big\{I(\varepsilon_{j,\alpha} \le  \Delta_\alpha^h(x,X_j) + um_0^{-1/2}) - \alpha \Big\}  \Big] \\
&=& um_0^{-1/2} E^j_x \Big[K\Big(\frac{x-X_j}{h}\Big) f_{\varepsilon_\alpha|X}(\Delta_\alpha^h(x,X_j)|X_j)  \Big]\\
&& + \frac12 u^2 m_0^{-1} E^j_x \Big[K\Big(\frac{x-X_j}{h}\Big) f'_{\varepsilon_\alpha|X}(\Delta_\alpha^h(x,X_j)|X_j)  \Big] \\
&& + O(L_n m_0^{-3/2} ),
\end{eqnarray*}
uniformly in $|u| \leq C^* L^*_n$, because of Assumption (B2).  Then, with
$$ \eta_{j,\alpha,u,x} = K\Big(\frac{x-X_j}{h}\Big) \Big\{I(\varepsilon^\Delta_{j,\alpha} \le   \Delta_\alpha^h(x)  +um_0^{-1/2}) - \alpha\Big\} - g_{x,\alpha}(u), $$
we have that
\begin{eqnarray*}
&& P\Big(\widehat r^{\Delta,-}_\alpha(x) \le um_0^{-1/2} \Big| {\cal N^{-}}(x), N^{-}(x)=m \Big) \\
&& = P \Big(m^{-1/2} \sum_{j \in {\cal N^{-}}(x)} \eta_{j,\alpha,u,x} \ge -m^{1/2}  g_{x,\alpha}(u) \Big|  {\cal N^{-}}(x), N^{-}(x)=m \Big). \end{eqnarray*}
We now argue that an Edgeworth expansion holds for the conditional density of $T_\eta=m^{-1/2} \sum_{j\in {\cal N^{-}}(x)} \eta_{j,\alpha,u,x}$, given  $ {\cal N^{-}}(x), N^{-}(x)=m$, that is of the form
\begin{eqnarray}\label{edge} \sigma^{-1} \sum_{r=0}^{s-3} m ^{-r/2}P_r(-\phi: \{ \chi_{\nu} \} ) ( \sigma^{-1} [\cdot - x])+ O(  m ^{-(s-2)/2} [1 + | \sigma^{-1} [\cdot - x ]|^s]^{-1})\end{eqnarray}
where, the error term holds  uniformly in $\alpha\in A$, $|u| \leq L_n^* $ and $x \in R_X$ for $C_1^* m_0\leq m \leq C_2^* m_0$ and constants $C_1^* < C_2^*$. Here, we use standard notation used e.g. in Bhattacharya and Rao (1976), p. 53. In particular, $\sigma^2 $ denotes the conditional variance of $\eta_{j,\alpha,u,x}$, given that $j \in {\cal N^{-}}(x)$, and $P_r(-\phi: \{ \chi_{\nu} \} ) $
denotes a product of a standard normal density $\phi$ with a polynomial that has coefficients
depending only on the conditional cumulants $\chi_\nu$ of $\eta_{j,\alpha,u,x}$ of order $\nu \leq s-1$, given that $j \in {\cal N^{-}}(x)$. Note that $\sigma^2 $ and $\chi_\nu$
depend on $u$, $\alpha$, $x$ and $n$ and that we do not indicate this in our notation. Furthermore, the cumulants and the variance converge to constants depending on $\alpha$, uniformly in $|u| \leq L_n^* $ and $x \in R_X$. Note that for $n \to \infty$ the conditional distribution of $\eta_{j,\alpha,u,x}$, given that $j \in {\cal N^{-}}(x)$, converges to the distribution of $K(U) (Z- \alpha)$ where $U$ and $Z$ are independent random variables, $U$ has a uniform distribution on $[-1,1]$ and $Z$ is $\{0,1\}$-valued with $P(Z=1) = \alpha$. This helps to understand that limit theorems hold uniformly. The function $P_r(-\phi: \{ \chi_{\nu} \} ) $ is defined as
$$P_r(-\phi: \{ \chi_{\nu} \} )  (u) = \sum_{m=1}^r \frac 1 {m!} (-1)^r\sum_{j_1+...+j_m=r} \frac{\chi_{j_1+2}}{(j_1+2)!}\cdot ... \cdot \frac{\chi_{j_m+2}}{(j_m+2)!} \phi^{(r+2m)}(u),$$
see Section 7 in in Bhattacharya
and Rao (1976).
In our case  expansion \eqref{edge} follows from Theorem 19.3 in Bhattacharya
and Rao (1976). For this claim we have to verify that their conditions (19.27), (19.29) and (19.30) hold. Our setting is slightly different from theirs, since we consider triangular arrays of independent identically distributed random variables instead of a sequence of independent random variables as is the case in Theorem 19.3 in Bhattacharya and
Rao (1976). But the same proof applies because
in our setting we can verify the following uniform versions of (19.27), (19.29) and (19.30):
\begin{eqnarray}\label{edgehelp1}  \sup_{\alpha\in A, |u| \leq L_n^*, x \in R_X, n \geq n_0}E_x^j \Big[|\eta_{j,\alpha,u,x}|^s \Big] &<& \infty, \\
\label{edgehelp2}  \sup_{\alpha\in A, |u| \leq L_n^*, x \in R_X, n \geq n_0} \int g_{\alpha,u,x}^q(t) \mathrm d t &<& \infty \mbox{ for some } q > 0, \\
\label{edgehelp3}   \sup_{\alpha\in A, |u| \leq L_n^*, x \in R_X, n \geq n_0}\{ g_{\alpha,u,x}(t): |t| \geq b\} &< &1 \mbox{ for all }  b > 0\end{eqnarray}
 with $g_{\alpha,u,x}(t) = | E^j_x [\exp(it \sigma^{-1} \eta_{j,\alpha,u,x})]|$ for some $n_0 > 0$. Note that $g_{\alpha,u,x}$, $\sigma$ and $ \eta_{j,\alpha,u,x}$ depend on $n$.

Claim \eqref{edgehelp1} follows by a direct argument using brute force bounds.
 For the proof of \eqref{edgehelp2}, we consider the conditional  density of $U_p= \sum_{j=1}^p (k(\frac{x_1-X_{1,j}}{h}), ...,k(\frac{x_d-X_{d,j}}{h}))^\intercal  \Big\{I(\varepsilon^\Delta_{j,\alpha} \le   \Delta_\alpha^h(x)  +um_0^{-1/2}) - \alpha\Big\}$  given the value of $\varepsilon^\Delta_{j,\alpha}$ and given that $X_j \in {\cal N^{-}}(x)$ for $j=1,...,p$. For $p=1$ this density evaluated at $(u_1,\ldots,u_d)$ can be bounded by a constant times $(u_1 \cdot ... \cdot u_d)^{-\kappa}((k(0) - u_1) \cdot ... \cdot(k(0) - u_d))^{-\kappa}$ by Assumptions (B1) and (B3). This bound holds uniformly over $\alpha$, $u$, $x$ and the value of $\varepsilon^\Delta_{1,\alpha}$.

 We now show that for every $\kappa^*> 0$, there exists $p^*>0$ such that  the density of $U_{p^*}$ can be bounded by a constant times $(u_1 \cdot ... \cdot u_d)^{\kappa^*}$. For simplification of notation we assume for the proof that $d=1$.
 For the proof of the claim, we show first that for $p \geq(1- \kappa)^{-1}$ we get that the conditional  density of $U_p$ is uniformly bounded. This follows by an evaluation of convolution integrals where one uses that $\int_0^v u^{-\kappa_1} (v-u) ^{-\kappa_2} \mathrm d u = v^{-(\kappa_1+\kappa_2-1)}\int_0^1 w^{-\kappa_1} (1-w) ^{-\kappa_2} \mathrm d w\leq C v^{-(\kappa_1+\kappa_2-1)}$. Applied to $U_2$ this gives that the density of $U_2$ is bounded by $C u_1 ^{-(2\kappa -1)}+ C |k(0) - u_1| ^{-(2\kappa -1)}  + C(2k(0) -u_1) ^{-(2\kappa -1)} $. Note that the density of $U_1$ is bounded by $Cu_1 ^{-\kappa } (k(0) - u_1)^{-\kappa } \leq Cu_1 ^{-\kappa } + C (k(0) - u_1)^{-\kappa }  $. For the density of $U_3$ we get the bound  $C u_1 ^{-(3\kappa -2)}+ C |k(0) - u_1| ^{-(3\kappa -2)} + C |2k(0) - u_1| ^{-(3\kappa -2)}  + C(3k(0) -u_1) ^{-(3\kappa -2)} $. Finally, for  $p \geq(1- \kappa)^{-1}$ it holds that $p \kappa- (p-1) \leq 0$ and we have that for  $p \geq(1- \kappa)^{-1}$, the density of $U_p$ is uniformly bounded. We now use that the $k$-fold convolution of a bounded density with bounded support $[0,z]$ for some $z> 0$ is bounded by $C |u_1|^{k-1}$. This gives that the density of $U_{p^*}$ can be bounded by a constant times $u_1 ^{\kappa^*}$ if $p^* \leq l_0 (\kappa^* +1)$ with $l_0 \geq(1- \kappa)^{-1}$, $ l \in \mathbb N$. This result can be easily extended to $d > 1$.

 From this result now we want to conclude that the conditional density   of $ \sum_{j=1}^p \eta_{j,\alpha,u,x}=\sum_{j=1}^p K(\frac{x-X_{j}}{h}) \Big\{I(\varepsilon^\Delta_{j,\alpha} \le   \Delta_\alpha^h(x) + um_0^{-1/2}) - \alpha\Big\}=\sum_{j=1}^p k(\frac{x_1-X_{1,j}}{h})\cdot ... \cdot k(\frac{x_d-X_{d,j}}{h}) \Big\{I(\varepsilon^\Delta_{j,\alpha} \le   \Delta_\alpha^h(x)  + um_0^{-1/2}) - \alpha\Big\}$ is uniformly bounded, given the value of $\varepsilon^\Delta_{j,\alpha}$ and given that $X_j \in {\cal N^{-}}(x)$ for $j=1,...,p$. This follows immediately from the following result. Suppose that $Z=(Z_1,...,Z_d)$ has support $[0,1]^d$ and a density $f$ that is bounded by
 $f(z) \leq D z_1 \cdot ... \cdot z_d$ then $Z_1 \cdot ... \cdot Z_d$ has a  density $g$ that is bounded by $D$.
 For a proof of this result, note that
 \begin{eqnarray*} g(u) &=& \partial_u \int_{z_1 \cdot ... \cdot z_d \leq u} f(z) \mathrm dz_1\   ... \ \mathrm dz_d\\ &=& \partial_u \int_{v \leq u} f\left ( \frac v {z_2\cdot ... \cdot z_d }, z_2,..., z_d\right) \frac 1 {z_2\cdot ... \cdot z_d } \mathrm d v\  \mathrm dz_2\   ...\ \mathrm dz_d\\ &=&  \int f\left ( \frac u {z_2\cdot ... \cdot z_d }, z_2,..., z_d\right) \frac 1 {z_2\cdot ... \cdot z_d }  \mathrm dz_2 \  ...\ \mathrm dz_d
 \\ &\leq&  \int f\left ( \frac u {z_2\cdot ... \cdot z_d }, z_2,..., z_d\right) \frac 1 {u }  \mathrm dz_2 \  ...\ \mathrm dz_d
  \\ &\leq&  \int D  \frac u {z_2\cdot ... \cdot z_d } \cdot  z_2 \cdot ... \cdot  z_d\cdot  \frac 1 {u }  \mathrm dz_2 \  ...\ \mathrm dz_d\\ &=& D.
 \end{eqnarray*}

 We now apply the result that for $p$ chosen large enough, the conditional density of  $\sum_{j=1}^p  \eta_{j,\alpha,u,x}$ is bounded, given that $X_j \in {\cal N^{-}}(x)$ for $j=1,...,p$, uniformly over $\alpha$, $u$ and $x$. This implies that the square of this conditional density is integrable and by the Fourier Inversion Theorem  (see Theorem 4.1 (vi) in Bhattacharya and
Rao (1976)) the same holds for the squared modulus of its Fourier transform. Thus the modulus of the Fourier transform of the conditional density of  $ \sum_{j=1}^{2p} \eta_{j,\alpha,u,x}$, given that $X_j \in {\cal N^{-}}(x)$ for $j=1,...,2p$, is integrable. This shows \eqref{edgehelp2} for $q=2p$.

For the proof of  \eqref{edgehelp3} one applies the Riemann-Lebesgue Lemma (see Theorem 4.1 in Bhattacharya and
Rao (1976)). Consider for simplicity the case where $d=1$.  For  $E_x^j [\exp(it  \eta_{j,\alpha,u,x})]$ one gets that \begin{eqnarray*} &&\exp[it g_{x,\alpha}(u)] E_x^j [\exp(it  \eta_{j,\alpha,u,x})]\\ && \qquad = \int _{x-h } ^{  x+h} \int_ {e \in
\mathbb{R}} \exp \Big [it K\left (\frac {x-z} h\right ) \{\mbox{I}(e \leq   \Delta_\alpha^h(x,z)  +u m_0^{-1/2} )- \alpha\} \Big] \\ && \qquad \qquad \times \ f_{\varepsilon_{\alpha}|X}(e |z) f_X(z)\ \mathrm d e \ \mathrm d z \bigg /  \int _{x-h } ^{  x+h} f_X(z)\mathrm d z\\
&& \qquad =  \int _{-1 } ^{  1} \int_ {e \in
\mathbb{R}} \exp \Big [it K\left (v\right ) \{\mbox{I}(e \leq   \Delta_\alpha^h(x,x+hv)  +u m_0^{-1/2} )- \alpha\} \Big] \\ && \qquad \qquad \times  \ f_{\varepsilon_{\alpha}|X}(e |x+hv) f_X(x+hv)\ \mathrm d e \ \mathrm d v  \bigg /  \int _{-1 } ^{  1} f_X(x+hv)\mathrm d v .\end{eqnarray*}
The right hand side of this equation converges to
\begin{eqnarray*}{\alpha} \int _{-1 } ^{  1 }  \exp[it (1 - \alpha) K(v)  ]dv + {(1-\alpha)} \int _{-1 } ^{  1 }  \exp[-it  \alpha K(v)  ]dv .\end{eqnarray*}
This convergence holds uniformly in $t \in \mathbb R$, $\alpha\in A$, $|u| \leq L_n^* m_0^{-1/2}$ and $x \in R_X$.
By using these facts we get  \eqref{edgehelp3}
from the Riemann-Lebesgue Lemma.

By applying Theorem 19.3 in Bhattacharya and
Rao (1976) with $s\geq 4$ we get that
\begin{eqnarray}
&& P\Big(\widehat r^{\Delta,-}_\alpha(x) \le um_0^{-1/2} \Big| {\cal N^{-}}(x), N^{-}(x)=m \Big)\label{edgeexp}\\ \nonumber
&& = 1 - \Phi\Big(\mu_\alpha(u)\Big) + m^{-1/2} \rho_\alpha(u) \Big(1-\mu_\alpha(u)^2\Big) \phi\Big(\mu_\alpha(u)\Big) + O\Big(m_0^{-1} (1+\mu_\alpha(u)^2)^{-s}\Big),
\end{eqnarray}
uniformly in $u$, $\alpha$ and $x$ for $C_1^* m_0\leq m \leq C_2^* m_0$ and constants $C_1^* < C_2^*$. Here we have used the fact that terms for $r= 2,..., s-3$ in the  expansion \eqref{edge} can be bounded by  $O\Big(m_0^{-1} (1+\mu_\alpha(u)^2)^{-s}\Big)$. We used the following notation
$$ \mu_\alpha(u) = -\frac{m^{1/2}  g_{x,\alpha}(u)}{\sigma_\alpha(u)} \:\: \mbox{ and } \:\: \rho_\alpha(u) = \frac {E_x^j(\eta_{j,\alpha,u,x}^3)}{6 \sigma_\alpha^3(u)}, $$
with $\sigma_\alpha^2(u) = E_x^j(\eta_{j,\alpha,u,x}^2)$. It is easy to show that, uniformly in $|u| \leq C^* L^*_n$,
\begin{eqnarray*}
 \sigma_\alpha^2(u) &=&    E_x ^j \Big[K^2\Big(\frac{x-X_i}{h}\Big)\Big(I(\varepsilon_{j,\alpha} \leq \Delta_\alpha^h (x, X_j)+ um_0^{-1/2} ) - \alpha\Big )^2 \Big]\\
 && \qquad \qquad+ O(L_n m_0^{-1} )\\&=&    A_1(\alpha) + um_0^{-1/2} A_2(\alpha)
 + O(L_n m_0^{-1} ),
\end{eqnarray*}
\begin{eqnarray*}
 E_x^j(\eta_{j,\alpha,u,x}^3) &=&
E_x ^j \Big[K^3\Big(\frac{x-X_i}{h}\Big)\Big(I(\varepsilon_{j,\alpha} \leq \Delta_\alpha^h (x, X_j)+ um_0^{-1/2} ) - \alpha\Big )^3 \Big]\\
 && \qquad \qquad+ O(L_n m_0^{-1/2} )\\
&& =A_3(\alpha)+ O(L_n m_0^{-1/2}),
\end{eqnarray*}
and that
\begin{eqnarray*}
\mu_\alpha(u)
&=& -um^{1/2}m_0^{-1/2} A_1(\alpha)^{-1/2} A_5(\alpha) - \frac12 u^2 m^{1/2}m_0^{-1} A_1(\alpha)^{-1/2} A_6(\alpha)\\
&& + \frac 1 2 u^2m^{1/2}m_0^{-1} A_1^{-3/2}(\alpha)A_2(\alpha) A_5(\alpha)+ O(L_n m_0^{-1} )
\end{eqnarray*}
with \begin{eqnarray*}
 A_1(\alpha) &=&    E_x ^j \Big[K^2\Big(\frac{x-X_i}{h}\Big)\Big((1-2\alpha) P(\varepsilon_{j,\alpha} \leq \Delta_\alpha^h (x, X_j) ) + \alpha^2\Big ) \Big],\\
  A_2(\alpha) &=& E_x ^j \Big[K^2\Big(\frac{x-X_i}{h}\Big)(1-2\alpha)  f_{\varepsilon_\alpha|X}( \Delta_\alpha^h (x, X_j) |X_j)\Big],\\
   A_3(\alpha) &=& E_x ^j \Big[K^3\Big(\frac{x-X_i}{h}\Big)\Big((1-3\alpha+ 3\alpha^2) P(\varepsilon_{j,\alpha} \leq \Delta_\alpha^h (x, X_j) ) - \alpha^3\Big ) \Big],\\
   A_4(\alpha) &=& E_x ^j \Big[K^2\Big(\frac{x-X_i}{h}\Big)\Big((1-2\alpha) P(\varepsilon_{j,\alpha} \leq \Delta_\alpha^h (x, X_j) ) + \alpha^2\Big ) \Big].,\\
   A_5(\alpha)&=& E^j_x \Big[K\Big(\frac{x-X_j}{h}\Big) f_{\varepsilon_\alpha|X}(\Delta_\alpha^h(x,X_j)|X_j)  \Big],\\
   A_6(\alpha)&=& E^j_x \Big[K\Big(\frac{x-X_j}{h}\Big) f'_{\varepsilon_\alpha|X}(\Delta_\alpha^h(x,X_j)|X_j)  \Big],\\
 \end{eqnarray*}
Note that $\mu_\alpha(-u)^2 = \mu_\alpha(u) ^2 + O(L_n  m_0^{-1/2})$. Thus we get that  uniformly in $|u| \leq C^* L^*_n$,
\begin{eqnarray}
\label{edga}
&& m^{-1/2} \rho_\alpha(u) \Big(1-\mu_\alpha(u)^2\Big) \phi\Big(\mu_\alpha(u)\Big)  - m^{-1/2} \rho_\alpha(-u) \Big(1-\mu_\alpha(-u)^2\Big) \phi\Big(\mu_\alpha(-u)\Big)  \\
\nonumber && \qquad= O(L_n  m_0^{-1} ),\\
\label{edgb}
&& m^{-1/2} \rho_\alpha(u) \Big(1-\mu_\alpha(u)^2\Big) \phi\Big(\mu_\alpha(u)\Big) = O(L_n m_0^{-1/2}  ).
\end{eqnarray}

Note also that with $u_m= u m^{1/2} m_0^{-1/2}$, uniformly in $|u| \leq C^* L^*_n$,
\begin{eqnarray*}
 1 - \Phi\Big(\mu_\alpha(u)\Big) &=& 1 - \Phi\Big(-u_m A_1^{-1/2}(\alpha)A_5(\alpha)\Big)\\
&&  + \phi \Big(-u_m A_1^{-1/2}(\alpha)A_5(\alpha)\Big) \frac{u_m^2}{2m^{1/2}}  (A_1^{-1/2}(\alpha)A_6(\alpha)\\ && \qquad \qquad - A_1^{-3/2}(\alpha)A_2(\alpha)A_5(\alpha)) + O(L_n  m_0^{-1} ).
\end{eqnarray*}
Hence, uniformly in $|u| \leq C^* L^*_n$,
\begin{eqnarray} \label{edg}
1-\Phi(\mu_\alpha(u)) + \Phi(-\mu_\alpha(-u)) &=& 2\Big[1-\Phi\Big(-u_m A_1^{-1/2}(\alpha)A_5(\alpha)\Big) \Big] + O(L_n  m_0^{-1} ).
\end{eqnarray}

{F}rom \eqref{edga}, (\ref{edg}) and the above calculations it now follows for $l \geq 1$  that with $D_m(\alpha)=m^{1/2} m_0^{-1/2} A_1^{-1/2}(\alpha)A_5(\alpha)$ and $ \widehat r^{\Delta,-}_\alpha(x) =  \widehat r^\Delta_\alpha(x) - \Delta_\alpha^h(x)$
\begin{eqnarray*}
&& E \Big\{m_0^{l}\widehat r^{\Delta,-}_\alpha(x)^{2l} I(|\widehat r^{\Delta,-}_\alpha(x)| \le  L^*_n m_0^{-1/2}) \Big| N^{-}(x) =m\Big\} \\
&& = 2l  \int_0^{L^*_n} v^{2l-1} P \Big(\widehat r^{\Delta,-}_\alpha(x) > vm_0^{-1/2} \Big| N^{-}(x)=m\Big) \, \mathrm dv \\
&& \qquad- {2l}   \int_{-L^*_n}^0 v^{{2l}-1}   P\Big(\widehat r^{\Delta,-}_\alpha(x) \le vm_0^{-1/2} \Big| N^{-}(x)=m \Big) \, \mathrm dv \\
&& = {2l}   \int_0^{L^*_n} v^{{2l}-1}  \Big[P \Big(\widehat r^{\Delta,-}_\alpha(x) > vm_0^{-1/2} \Big| N^{-}(x)=m\Big) \\
&& \qquad + P \Big(\widehat r^{\Delta,-}_\alpha(x) \le -vm_0^{-1/2} \Big| N^{-}(x) =m\Big)\Big] \,\mathrm dv \\
&& = {2l}  \int_0^{L^*_n} v^{{2l}-1}  \Big[ \Phi\Big(\mu_\alpha(v)\Big) - m^{-1/2} \rho_\alpha(v) \Big(1-\mu_\alpha(v)^2\Big) \phi\Big(\mu_\alpha(v)\Big)\\
&& \qquad +1- \Phi\Big(\mu_\alpha(-v)\Big) + m^{-1/2} \rho_\alpha(-v) \Big(1-\mu_\alpha(-v)^2\Big) \phi\Big(\mu_\alpha(-v)\Big)\Big] \mathrm d v + O(L_n m_0^{-1})
 \\
&& = 4{l} \int_0^{L^*_n} v^{{2l}-1}   \Phi\Big(-v D_m(\alpha)\Big)  \, \mathrm dv  + O(L_n m_0^{-1})\\
&& = 4 {l} D_m(\alpha)^{-{2l}} \int_0^{L^*_nD_m(\alpha)} w^{{2l}-1}   \Phi(-w )  \, \mathrm dw  + O(L_n m_0^{-1})\\
&& = 2  \Big[(L^*_n)^{2l} \Phi\Big(-L^*_n D_m(\alpha)\Big) + D_m(\alpha)^{-{2l}} \int_0^{L^*_n D_m(\alpha)} v^{2l}  \phi(v) \, dv \Big]+ O(L_n m_0^{-1})\end{eqnarray*}
uniformly in $C_1^* m_0\leq m \leq C_2^* m_0$ with constants $C_1^* < C_2^*$.
If $L_n^* =  (\log n)^{\gamma}$ is chosen with $\gamma > 0$ large enough we get that the right hand side of the last equation is equal to $D_m(\alpha)^{-{2l}} \int_{-\infty}^{\infty} v^{2l}  \phi(v) \, dv + O(L_n m_0^{-1})$. This follows since it can be easily shown that \begin{eqnarray*}
&&
2\int_0^{L^*_n D_m(\alpha)} z^{2l} \phi(z) \, dz - \int_{-\infty}^\infty z^{2l} \phi(z) \, dz = o(L_n n^{-C^*}) = o(m^{-C^*}),\\
&&(L^*_n)^{2l} \Phi\Big(-L^*_n D_m(\alpha)\Big)= o(m^{-C^*})
\end{eqnarray*} with $\gamma $ chosen depending on $C^*$. Thus we get for  $l \in \mathbb N$ that
\begin{eqnarray}\label{remain0}
&& E \Big\{\widehat r^{\Delta,-}_\alpha(x)^{2l} I(|\widehat r^{\Delta,-}_\alpha(x)| \le  L^*_n m_0^{-1/2}) \Big| N^{-}(x) =m\Big\} \\ \nonumber
&& = m^{-l} \frac{A_1^{l}(\alpha)}{A_5^{2l}(\alpha)}  \int_{-\infty}^\infty z^{2l} \phi(z) \, dz + O(L_n (nh^d)^{-l-1}  ).
\end{eqnarray} With similar arguments one can show that
\begin{eqnarray}\label{remain0add}
&& E \Big\{\widehat r^{\Delta,-}_\alpha(x)^{2l-1} I(|\widehat r^{\Delta,-}_\alpha(x)| \le  L^*_n m_0^{-1/2}) \Big| N^{-}(x) =m\Big\} \\ \nonumber
&& = m^{-(2l-1)/2} \frac{A_1^{(2l-1)/2}(\alpha)}{A_5^{2l-1}(\alpha)}  \int_{-\infty}^\infty z^{2l-1}\phi(z) \, dz + O(L_n (nh^d)^{-l}  )\\ \nonumber
&& = O(L_n (nh^d)^{-l}  ).
\end{eqnarray}
 In this case one applies  \eqref{edgb} instead of  \eqref{edga}.

For the proof of (\ref{July3}) and (\ref{July3add1}) it remains to show that uniformly in $C_1^* m_0\leq m \leq C_2^* m_0$, with constants $C_1^* < C_2^*$, and
for  $l \in \mathbb N$
\begin{eqnarray}\label{remain1}
&& E \Big\{\widetilde r^{\Delta,-}_\alpha(x)^{2l}\Big| N^{-}(x) =m\Big\} \\ \nonumber
&& = m^{-\kappa/2} \frac{A_1^{\kappa/2}(\alpha)}{A_5^\kappa(\alpha)}  \int_{-\infty}^\infty z^{2l} \phi(z) \, dz + O(L_n (nh^d)^{-l-1}  ),
\\
\label{remain2}
&& E \Big\{\widetilde r^{\Delta,-}_\alpha(x)^{2l-1} \Big| N^{-}(x) =m\Big\} \\ \nonumber
&& = O(L_n (nh^d)^{-l}  ).
\end{eqnarray}

It remains to show \eqref{remain1} -- \eqref{remain2}. For the proof of \eqref{remain1}  note that for independent random variables $Z_1,...,Z_m$ with mean zero, variance 1 and bounded $2l$-th absolute moment it holds that 
$$E\Big\{(m^{-1/2}\sum_{i=1}^m Z_i )^{2l} \Big\} = \int_{-\infty}^\infty z^{2l} \phi(z) \, dz+O(m^{-1}), $$
because for $Z_1^*,\ldots,Z_m^* \stackrel{iid}{\sim} N(0,1)$ one has
\begin{eqnarray*}
E\Big\{(m^{-1/2}\sum_{i=1}^m Z_i )^{2l} \Big\} &=& m^{-l} \textstyle{\sum^*} E(Z_{i_1} \ldots Z_{i_{2l}}) + O(m^{-1}) \\ 
&=& m^{-l} \textstyle{\sum^*} E(Z^*_{i_1} \ldots Z^*_{i_{2l}}) + O(m^{-1}) \\
&=& E\Big\{(m^{-1/2}\sum_{i=1}^m Z_i^*)^{2l} \Big\}+O(m^{-1}) \\
&=& E((Z_1^{*})^{2l}) + O(m^{-1})\\
&=& \int_{-\infty}^\infty z^{2l} \phi(z) \, dz+O(m^{-1}), 
\end{eqnarray*}
where the sum $\textstyle{\sum}^*$ runs over all indices $i_1,\ldots,i_{2l}$ that are such that each value of an index appears exactly two times.  
%\int_{-\infty}^\infty z^{2l} \phi(z) \, dz+O(m^{-1}).$$
For the proof of \eqref{remain2}  one applies that for independent random variables $Z_1,...,Z_m$ with mean zero, variance 1 and bounded $2l+1$-th absolute moment it holds that 
\begin{eqnarray*}
E\Big\{(m^{-1/2}\sum_{i=1}^m Z_i )^{2l-1} \Big\} &=& m^{-l+1} m^{-1/2} E(\sum_{i=1}^m Z_i)^{2l-1} \\
&=& m^{-l+1} m^{-1/2} \textstyle{\sum^{**}} E(Z_{i_1} \ldots Z_{i_{2l-1}}) + O(m^{-3/2}) = O(m^{-1/2}), 
\end{eqnarray*}
%= \int_{-\infty}^\infty z^{2l+1} \phi(z) \, dz+O(m^{-1/2}).$$
where the sum $\textstyle{\sum^{**}}$ runs over all indices that are such that one value of an index appears three times and for all other $2l-4$ indices each value appears exactly two times.
This concludes the proof of the theorem.

\section{Proof of Theorem \ref{test}}

For the proof of Theorem \ref{test} we will use the following corollary of Theorem \ref{theolem4}.
For the statement of the corollary we have to define another construction of local neighborhoods. For their  definition suppose first that $X$ is one-dimensional. Then the support $R_X$ is a compact interval. For arbitrary $j$ and for $k\in \{1,2,3\}$, we can then define
$$ I_{jk} = [(3j+k-1)h,(3j+k)h], \hspace*{1cm} \mbox{and} \hspace*{1cm} I_{jk}^* = [(3j+k-2)h,(3j+k+1)h]. $$
The set of indices of the $X_i$ ($i=1,\ldots,n$) that fall inside the interval $I_{jk}^*$ is denoted by ${\cal N}_{jk}$. We write $N_{jk}$ for the number of elements of ${\cal N}_{jk}$.  An arbitrary $x \in R_X$ belongs to a unique $I_{jk}$ and we define ${\cal N}(x) = {\cal N}_{jk}$ and $N(x) = N_{jk}$.  Thus ${\cal N}(x)$ is an interval of length $3h$, such that $x$ lies in the middle subinterval of ${\cal N}(x)$ of length $h$. If the dimension of $X$ is larger than one, this partition of the support into small intervals can be generalized in an obvious way.

\begin{corollary} \label{corlem4}
Assume (B1)--(B3). Then, for natural numbers $l \geq 1$,
\begin{eqnarray*}   E\Big\{\overline r_{\alpha}^{\Delta,-}(x)^{2l} - \widetilde r_{\alpha}^{\Delta,-}(x)^{2l} \Big| N(x)=m \Big\} = O(L_n (nh^d)^{-l-1}  ) ,\\
 E\Big\{\overline r_{\alpha}^{\Delta,-}(x)^{2l-1} - \widetilde r_{\alpha}^{\Delta,-}(x)^{2l-1} \Big| N(x)=m \Big\} = O(L_n (nh^d)^{-l}  ) ,\end {eqnarray*}
uniformly in $x \in R_X$, $\alpha \in A$ and $C_1^* nh^d\leq m \leq C_2^* nh^d$,  where $N(x)$ is the random number of $X_i$'s that lie in ${\cal N}(x)$, and where
$ \overline r^{\Delta,-}_\alpha(x) =  \overline r^\Delta_\alpha(x) - \Delta_\alpha^h(x)$. For the second moments of  the uncentered estimators $\overline r_{\alpha}^{\Delta}$ and $ \widetilde r_{\alpha}^{\Delta}$  we have that
\begin{eqnarray*}   E\Big\{\overline r_{\alpha}^{\Delta}(x)^{2} - \widetilde r_{\alpha}^{\Delta}(x)^{2} \Big| N(x)=m \Big\} = O(L_n n^{-3/2}h^{-5d/4}   ).
\end {eqnarray*}
Under the additional assumption that $\Delta_\alpha \equiv 0$ we get that 
\begin{eqnarray*} E\Big\{\overline r_{\alpha}^{\Delta}(x)^{2} - \widetilde r_{\alpha}^{\Delta}(x)^{2} \Big| N(x)=m \Big\}  = O(L_n (nh^d)^{-2}  ).
\end {eqnarray*}

 \end{corollary}

\begin{proof}[Proof of Corollary \ref{corlem4}] For $m^{+} \geq m$ we have by a simple argument with $\kappa = 2l$ or $\kappa = 2l+1$ that $ E\Big\{\overline r_{\alpha}^{\Delta,-}(x)^\kappa - \widetilde r_{\alpha}^{\Delta,-}(x)^\kappa\Big| N(x)=m^{+} , N^{-}(x)=m \Big\} = E\Big\{\overline r_{\alpha}^{\Delta,-}(x)^\kappa - \widetilde r_{\alpha}^{\Delta,-}(x)^\kappa \Big| N^{-}(x)=m \Big\}$. Note that $N^{-}(x) \leq N(x)$ because of ${\cal N}^{-}(x) \subset {\cal N}(x)$.
Using  (\ref{July3})  and
$$P \left( N^{-}(x) \leq {m^+\over 4} \Big | N(x) =m^+\right ) \leq C \exp(-cnh^d),$$
uniformly in $m^+ \geq {1\over 2} 3^d f_X(x) nh^d$ we conclude that
$$ E\Big\{\overline r_{\alpha}^{\Delta,-}(x)^\kappa - \widetilde r_{\alpha}^{\Delta,-}(x)^\kappa \Big| N(x)=m^{+}  \Big\}  = O(L_n (nh^d)^{-l-1}),$$
uniformly in $x \in R_X$, $\alpha \in A$ and $ {1\over 2} 3^d f_X(x) nh^d \leq m^+ \leq 2\  3^d f_X(x) nh^d$.

Since
$$P\left ( {1\over 2} 3^d f_X(x) nh^d \leq N(x) \leq 2\  3^d f_X(x) nh^d \mbox { for all } x \in R_X\right ) \to 1, $$ we get the statement of the corollary.
\end{proof}

We now come to the proof of Theorem \ref{test}.

We only prove the statement for $ \widehat T_A$. The asymptotic result for $ \widehat T_\alpha$ follows similarly.  We need to introduce a few more notations. With $\delta_{\theta,\alpha} (x) = -(\theta(\alpha)-\theta_0(\alpha))^\top \gamma_\alpha(x)+ n^{-1/2} h^{-d/4} \Delta_\alpha(x)$ and $\varepsilon_{i,\alpha}^\Delta= \varepsilon_{i,\alpha} + n^{-1/2} h^{-d/4} \Delta_\alpha(X_i)$
we define $ \widetilde r^\Delta_{\alpha}$ as in \eqref{rtilde} and we put
\begin{eqnarray*}
 \widehat r^\Delta_{\alpha,\theta}(x) &= & \arg \min_r \sum_{i=1}^n K\left ( {x-X_i \over h}\right ) \tau_{\alpha} (\varepsilon_{i.\alpha} +\delta_{\theta,\alpha} (X_i)-r).
\end{eqnarray*}
	Note that $\widehat r^\Delta_{\alpha}(x) =  \widehat r^\Delta_{\alpha,\theta_0}(x)$, and that $\widehat r_{\alpha}(x) = \widehat r^\Delta_{\alpha,\widehat \theta}(x) + O_P(n^{-1/2-c})$ by Assumption (B4). We also define $ \overline r_{\alpha}^\Delta$ as in  \eqref{rtildeadd2}.
Let also
\begin{eqnarray*}
W_{ni}(x,h) &=& K_h(x-X_i)/[\sum_j K_h(x-X_j)],
%W^{f,\alpha}_{ni}(x,h) &=& K_h(x-X_i) f_{\varepsilon_\alpha|X}(0|X_i) /[\sum_j K_h(x-X_j)f_{\varepsilon_\alpha|X}(0|X_j) ]
\end{eqnarray*}
with $K_h(\cdot) = K(\cdot/h)/h^d$.

The proof of Theorem \ref{test} will make use of the following lemmas.
\begin{lemma} \label{lem1}
Suppose that the assumptions of Theorem \ref{test} are satisfied. Then,
\begin{eqnarray}
&&\sup_{\alpha \in A} \sup_{x \in R_X} \Big|\widehat r_{\alpha}(x)  \Big| = O_P((nh^d)^{-1/2} L_n), \label{eq1}\\
&&\sup_{\alpha \in A} \sup_{x \in R_X} \Big|\widehat r^\Delta_\alpha(x)  \Big| = O_P((nh^d)^{-1/2} L_n) \label{eq2}.
\end{eqnarray}
\end{lemma}

\begin{proof}[Proof of Lemma \ref{lem1}] As is known for the case where there is no parametric part and where $\Delta_\alpha\equiv 0$, one has that $$ \sup_{\alpha \in A} \sup_{x \in R_X} \Big|\widehat r_{\alpha,\theta_0}^{\Delta}(x)  -\widetilde r_\alpha(x) \Big| = O_P((nh^d)^{-3/4} L_n)$$
with $\widetilde r_\alpha$ defined as in \eqref{addhd}.
For a proof see Theorem 2   in Guerre and Sabbah (2012). By standard smoothing theory we have that (still when $\Delta_\alpha \equiv 0$)
\begin{eqnarray} \label{tildestar}
\sup_{\alpha \in A} \sup_{x \in R_X} \Big|\widetilde r_\alpha(x) \Big| = O_P((nh^d)^{-1/2} L_n).
\end{eqnarray}
This shows (\ref{eq2}) when $\Delta_\alpha \equiv 0$.  
We can move from this case to  $\Delta_\alpha \not = 0$ by adding to the observations  terms of order $O_P(n^{-1/2} h^{-d/4})$.  This changes the local quantiles  by at most this amount, and hence (\ref{eq2}) still holds when $\Delta_\alpha \neq 0$.  

In the case of $\widehat r_{\alpha}(x)=\widehat r_{\alpha, \widehat\theta}^\Delta(x)+O_P(n^{-1/2-c})$ we have to add to the observations terms of the order $ O_P(L_nn^{-1/2} h^{-d/4})= O_P((nh^d)^{-1/2} L_n)$. This shows the first statement of the lemma. \end{proof}

\begin{lemma} \label{lem2a}
Suppose that the assumptions of Theorem \ref{test} are satisfied. Then,
\begin{eqnarray*}
&&\sup_{\alpha \in A} \sup_{x \in R_X} \Big|\widehat r_\alpha(x) - \widehat r_\alpha^\Delta(x) + (\widehat \theta(\alpha)
-  \theta_0(\alpha))^\top \gamma_\alpha(x) \Big| = O_P(n^{-{1 \over 2} -c} ).
\end{eqnarray*}
\end{lemma}

\begin{proof}[Proof of Lemma \ref{lem2a}]
First note that $\widehat r_\alpha(x) + (\widehat \theta(\alpha) - \theta_0(\alpha))^\top \gamma_\alpha(x)$ is equal to the quantile estimator we would obtain when we shift all observations $Y_i$ in the window around $x$ by the amount $(\widehat \theta(\alpha) - \theta_0(\alpha))^\top \gamma_\alpha(x)$, and hence we need to show that the distance between this latter estimator (say $\widehat r_{\alpha,mod}(x)$) and $\widehat r_\alpha^\Delta(x)$ is $O_P(n^{-{1 \over 2} -c})$ uniformly in $\alpha$ and $x$.

Next, note that if now in addition we perturb all observations in the window around $x$ by adding $m_{\alpha,\widehat \theta(\alpha)}(X_i) - m_{\alpha,\theta_0(\alpha)}(X_i) - (\widehat \theta(\alpha) - \theta_0(\alpha))^\top \gamma_\alpha(X_i)$, the quantile estimator $\widehat r_{\alpha,mod}(x)$ will get perturbed by at most the maximal perturbation of the observations, which is of the order $O_P(n^{-1/2-c})$ by Assumption (B4).

After these two perturbations, the quantile estimator is now based on $Y_i - m_{\alpha,\theta_0(\alpha)}(X_i) + (\widehat \theta(\alpha) - \theta_0(\alpha))^\top (\gamma_\alpha(x) - \gamma_\alpha(X_i))$ instead of $Y_i - m_{\alpha,\widehat \theta(\alpha)}(X_i)$.  Finally note that if we apply one more perturbation by subtracting $(\widehat \theta(\alpha) - \theta_0(\alpha))^\top (\gamma_\alpha(x) - \gamma_\alpha(X_i))$ for all $X_i$ in the window around $x$, the estimator changes by at most $O_P(h^{-\delta} n^{-1/2-\rho}h^\delta) = O_P(n^{-1/2-\rho})$ by Assumption (B5).   The so-obtained estimator equals $\widehat r_\alpha^\Delta(x)$, which shows the statement of the lemma.
\end{proof}

\begin{lemma} \label{lem2}
Suppose that the assumptions of Theorem \ref{test} are satisfied. Then,
\begin{eqnarray*} &&\sup_{\alpha \in A} \sup_{x \in R_X} \Big|\widehat r_\alpha^\Delta(x) - \widetilde r^\Delta_\alpha(x) \Big| = O_P((nh^d)^{-3/4} L_n).\end{eqnarray*}
\end{lemma}

\begin{proof}[Proof of Lemma \ref{lem2}]
Write
\begin{eqnarray}
&& |\widehat r^\Delta_\alpha(x) - \widetilde r^\Delta_\alpha(x)| \nonumber \\
&& \le \frac{1}{\inf_{x,\alpha} f_{\varepsilon_\alpha|X}(0|x)} \Big|\sumi W_{ni}(x,h) f_{\varepsilon_\alpha|X}(0|X_i)
\widehat r^\Delta_\alpha(x) + \sumi W_{ni}(x,h) \big(I(\varepsilon^\Delta_{i,\alpha} \le 0) - \alpha \big) \Big| \nonumber \\
&& =  \frac{1}{\inf_{x,\alpha} f_{\varepsilon_\alpha|X}(0|x)} \Big|\sumi W_{ni}(x,h) f_{\varepsilon_\alpha|X}(0|X_i) \widehat r^\Delta_\alpha(x)  - \widehat F_{\varepsilon^\Delta_\alpha|X}(\widehat r^\Delta_\alpha(x)
|x) + \widehat F_{\varepsilon^\Delta_\alpha|X}(0|x) \Big| \nonumber \\
&& \qquad  + O_P((nh^d)^{-1} ),  \label{form}
\end{eqnarray}where $\widehat F_{\varepsilon^\Delta_\alpha|X}(y|x) = \sum_i W_{ni}(x,h) I(\varepsilon^\Delta_{i,\alpha} \le y)$.  The latter equality follows from the fact that
\begin{eqnarray*}
|\widehat F_{\varepsilon^\Delta_\alpha|X}(\widehat r^\Delta_\alpha(x)|x) - \alpha| & \le & |\widehat F_{\varepsilon^\Delta_\alpha|X}(\widehat r^\Delta_\alpha(x)|x) - \widehat F_{\varepsilon^\Delta_\alpha|X}(\widehat r^\Delta_\alpha(x)-|x) | \\
& = & O_P((nh^d)^{-1} ).
\end{eqnarray*}
The following expansion follows from standard kernel smoothing theory, uniformly for $x \in R_X, \alpha \in A, |y|\le a_n$ and for sequences $a_n$ with $a_n^{-1}=  O(nh^d)$ :
\begin{eqnarray*}
&& \widehat F_{\varepsilon^\Delta_\alpha|X}(y|x) -  \widehat F_{\varepsilon^\Delta_\alpha|X}(0|x)
\nonumber \\
&& \qquad =  \sum_i W_{ni}(x,h) \int_0^y f_{\varepsilon_\alpha|X}(u-n^{-1/2} h^{-d/4} \Delta_\alpha(X_i)|X_i) du + O_P((nh^d)^{-1/2} L_n a_n^{1/2}) \\
&& \qquad =  \sum_i W_{ni}(x,h) \int_0^y f_{\varepsilon_\alpha|X}(u|X_i) du + O_P((nh^d)^{-1/2} L_n a_n^{1/2}) + O_P(n^{-1/2} h^{-d/4} a_n) \\
&& \qquad =  y \sum_i W_{ni}(x,h) f_{\varepsilon_\alpha|X}(0|X_i)  + O_P((nh^d)^{-1/2} L_n a_n^{1/2}+ a_n^2) + O_P(n^{-1/2} h^{-d/4} a_n).
\end{eqnarray*}
We now apply this bound to $a_n = (nh^d)^{-1/2} L_n$ and $y = \widehat r^\Delta_\alpha(x)$, which is possible thanks to Lemma \ref{lem1}.  This combined with (\ref{form}) shows the statement of the lemma.
\end{proof}

\smallskip
\bigskip

For proving Theorem \ref{test}, we will make use of the following decomposition, which follows from Lemma \ref{lem2a} :
\begin{eqnarray*}
\widehat T_A  \hspace*{-.5cm} && =   \int_A \int_{R_X} \Big[  \widehat r^\Delta_\alpha(x)-  (\widehat \theta(\alpha)
-  \theta_0(\alpha))^\top \gamma_\alpha(x) \Big]^2 w(x,\alpha) \, dx \, d\alpha + o_P(n^{-1} h^{-d/2})\\
&& =   \int_A \int_{R_X} \Big[  \widehat r^\Delta_\alpha(x)^2-\overline r^\Delta_{\alpha}(x) ^2\Big] w(x,\alpha) \, dx \, d\alpha \\
&& \qquad + \int_A \int_{R_X} E\Big \{\overline r^\Delta_{\alpha}(x) ^2- \widetilde r^\Delta_{\alpha}(x) ^2\Big| N(x) \Big\} w(x,\alpha) \, dx \, d\alpha \\
&& \qquad + \int_A \int_{R_X}  \Big[ \overline r^\Delta_{\alpha}(x) ^2- \widetilde r^\Delta_{\alpha}(x) ^2  -E\Big \{\overline r^\Delta_{\alpha}(x) ^2- \widetilde r^\Delta_{\alpha}(x) ^2\Big| N(x) \Big\} \Big] w(x,\alpha) \, dx \, d\alpha \\
&& \qquad - 2   \int_A \int_{R_X} \Big[  (\widehat r^\Delta_\alpha(x)- \widetilde r^\Delta_{\alpha}(x)) \big\{ (\widehat \theta(\alpha)
-  \theta_0(\alpha))^\top \gamma_\alpha(x) \big\}\Big] w(x,\alpha) \, dx \, d\alpha \\
&& \qquad - 2   \int_A \int_{R_X} \Big[   \widetilde r^\Delta_{\alpha}(x)\big\{ (\widehat \theta(\alpha)
-  \theta_0(\alpha))^\top \gamma_\alpha(x) \big\}\Big] w(x,\alpha) \, dx \, d\alpha \\
&& \qquad +   \int_A \int_{R_X} \Big[  (\widehat \theta(\alpha)
-  \theta_0(\alpha))^\top \gamma_\alpha(x) \Big]^2 w(x,\alpha) \, dx \, d\alpha \\
&& \qquad + \int_A \int_{R_X}  \widetilde r^\Delta_{\alpha}(x)^2 w(x,\alpha) \, dx \, d\alpha + o_P(n^{-1} h^{-d/2})\\
&& = T_{n1} + ... + T_{n7} + o_P(n^{-1} h^{-d/2}).
\end{eqnarray*}

\begin{lemma} \label{lem3}
Suppose that the assumptions of Theorem \ref{test} are satisfied. Then,
$$ T_{n1} = o_P(a_n), $$
for any sequence $\{a_n\}$ of positive constants tending to zero as $n \rightarrow \infty$.
\end{lemma}

\begin{proof}[Proof of Lemma \ref{lem3}] Note that
\begin{eqnarray*}
T_{n1} \hspace*{-.5cm} && \le \sup_{\alpha \in A} \sup_{x \in R_X} |\widehat r^\Delta_\alpha(x)|^2 \int_A \int_{R_X} I \Big(|\widehat r^\Delta_\alpha(x)| > L_n (nh^d)^{-1/2} \Big) w(x,\alpha) \, dx \, d\alpha.
\end{eqnarray*}
It is easily seen from Lemma \ref{lem1} that
$$ \int_A \int_{R_X} I \Big(|\widehat r^\Delta_\alpha(x)| > L_n (nh^d)^{-1/2} \Big) w(x,\alpha) \, dx \, d\alpha  = o_P(a_n), $$
for any $a_n \rightarrow 0$, since the indicator inside the integral will be zero from some point on.  \end{proof}

From Corollary \ref{corlem4} we get the following result.
\begin{lemma} \label{lem4}
Suppose that the assumptions of Theorem \ref{test} are satisfied.  Then,
\begin{eqnarray*}
\sup_{\alpha \in A} \sup_{x \in R_X}  \Big|E\Big\{\overline r^\Delta_{\alpha}(x)^2 - \widetilde r^\Delta_{\alpha}(x)^2 \Big| N(x) \Big\} \Big| = o_P((nh^{d/2})^{-1}),
\end{eqnarray*}
and hence, $
T_{n2} =  o_P((nh^{d/2})^{-1}).$%provided $nh^{3d/2} \rightarrow \infty$.
\end{lemma}

At this point we needed the additional assumption $n h^{3d/2}/L_n^*\to \infty$ for the case that $\Delta_\alpha\not \equiv 0$. We now shortly outline what happens if we are on the alternative and if this assumption does not hold. Note that for $|m-m_0| = o(m_0)$ \begin{eqnarray*} &&E\Big\{\overline r^\Delta_{\alpha}(x)^2 -\widetilde r^\Delta_{\alpha}(x)^2 \Big| N(x)=m \Big\} \\ &&=  E\Big\{(\overline r^{\Delta,-}_{\alpha}(x)+\Delta_\alpha^h(x)) ^2 -(\widetilde r^{\Delta,-}_{\alpha}(x)+\Delta_\alpha^h(x))^2 \Big| N(x)=m \Big\}  \\ &&=  E\Big\{\overline r^{\Delta,-}_{\alpha}(x)^2 -\widetilde r^{\Delta,-}_{\alpha}(x)^2 \Big| N(x)=m \Big\} + 2  \Delta_\alpha^h(x) E\Big\{\overline r^{\Delta,-}_{\alpha}(x) \Big| N(x)=m \Big\} .
\end{eqnarray*}
For the first term on the right hand side we get from Corollary \ref{corlem4} that it is of order $o((nh^{d/2})^{-1})$. For $\Delta_\alpha^h(x)$ one can show that it is equal to $n^{-1/2} h^{-d/4} \Delta_\alpha(x) + O(L_n n^{-1/2} h^{-d/4} h^2)$. For the term $E\Big\{\overline r^{\Delta,-}_{\alpha}(x) \Big| N(x)=m \Big\}$ one can show that it is equal to $(nh^d)^{-1} \rho(x) + O(L_n (nh^d)^{-3/2})$ for some function $\rho$ that does not depend on  the function $\Delta_\alpha$. This can be done by using the arguments based on Edgeworth expansions that were central in the proof of Theorem  \ref{theolem4}. This gives that 
$$T_{n2} = n^{-3/2} h^{-5d/4}  \int_A \int_{R_X} \Delta_\alpha(x)\rho(x) w(x,\alpha) \, dx \, d\alpha + o(n^{-3/2} h^{-5d/4}) + o((nh^{d/2})^{-1}).$$
Suppose now that $n h^{3d/2} \to 0$. Then it holds that  $ (nh^{d/2})^{-1}= o(n^{-3/2} h^{-5d/4} )$ and 
using Lemma \ref {lem9} we get that 
\begin{eqnarray*} T_n &=& n^{-1} h^{-d} K^{(2)}(0)  \int_A \alpha(1-\alpha) \int_{R_X} \frac{w(x,\alpha)}{f_X(x) f^2_{\varepsilon_\alpha|X}(0|x)} dx \, d\alpha\\ && \qquad  + n^{-3/2} h^{-5d/4}  \int_A \int_{R_X} \Delta_\alpha(x)\rho(x) w(x,\alpha) \, dx \, d\alpha + o_P(n^{-3/2} h^{-5d/4}).\end{eqnarray*}
This implies that the test rejects for large values of $ \int_A \int_{R_X} \Delta_\alpha(x)\rho(x) w(x,\alpha) \, dx \, d\alpha$. Thus, in this high-dimensional setting the test behaves like a linear test and not like an omnibus test.

\begin{lemma} \label{lem5}
Suppose the assumptions of Theorem \ref{test} are satisfied. Then,
\begin{eqnarray*}
T_{n3} = O_P( L_n n^{-5/4} h^{-3d/4}) =  o_P((nh^{d/2})^{-1}).
\end{eqnarray*}
\end{lemma}

\begin{proof}[Proof of Lemma \ref{lem5}] For simplicity of exposition of the argument, let us assume that $X_i$ is one-dimensional. For arbitrary $j$ and for $k\in\{1,2,3\}$, define
$$U_{jk} = \int_A \int_{I_{jk}} \Big[\overline r^\Delta_{\alpha}(x)^2 -\widetilde r^\Delta_{\alpha}(x)^2 - E\Big\{\overline r^\Delta_{\alpha}(x)^2 -\widetilde r^\Delta_{\alpha}(x)^2 \Big| N(x) \Big\} \Big] w(x, \alpha) \, dx \, d\alpha. $$
Then we can write $T_{n3} = T_{n31} + T_{n32} + T_{n33}$ with $T_{n3k} = \sum_j U_{jk}$ ($k=1,2,3$).  The terms $T_{n31}$, $T_{n32}$ and $T_{n33}$ are sums of $O(h^{-1})$ conditionally independent summands. The summands are uniformly bounded by a term of order $O_P(L_n n^{-5/4} h^{-1/4})$.  This follows from Lemma \ref{lem4}, from the fact that $\sup_{\alpha \in A}\sup_x|\widetilde r^\Delta_\alpha(x)| = O_P(L_n(nh)^{-1/2})$, see also (\ref{tildestar}), and from the Bahadur representation for $\overline r^\Delta_{\alpha}(x)$, given in Lemma \ref{lem2}.   It now follows that $T_{n3k}=O_P(L_n n^{-5/4} h^{-3/4})$, which implies the statement of the lemma for $d=1$. For $d>1$ one can use the same approach.  \end{proof}

\begin{lemma} \label{lem6}
Suppose the assumptions of Theorem \ref{test} are satisfied. Then,
\begin{eqnarray*}
T_{n4} =    o_P((nh^{d/2})^{-1}).
\end{eqnarray*}
\end{lemma}

\begin{proof}[Proof of Lemma \ref{lem6}]  This is obvious, since $T_{n4} = O_P(L_n (nh^d)^{-3/4} n^{-{1\over 4}}h^{{d\over 4}}/L_n^*)=o_P((nh^{d/2})^{-1})$, thanks to Assumption (B5) and Lemma \ref{lem2}.
\end{proof}

\begin{lemma} \label{lem7}
Suppose the assumptions of Theorem \ref{test} are satisfied. Then,
\begin{eqnarray*}
T_{n5} =   o_P((nh^{d/2})^{-1}).
\end{eqnarray*}
\end{lemma}

\begin{proof}[Proof of Lemma \ref{lem7}]
%We start by showing that
%\begin{eqnarray} \label{lem7claim} \int_A \int_{R_X}   \widetilde r_{\alpha}(x)( m_{\alpha,\widehat \theta(\alpha)}(x)
%-m_{\alpha,0}(x))w(x,\alpha) \, dx \, d\alpha= o_P((nh^{d/2})^{-1}). \end{eqnarray}
%The left hand side of (\ref{lem7claim}) is equal to
Write
\begin{eqnarray}
T_{n5} &=& 2\int_A \int_{R_X}   {\sumi K\Big(\frac{x-X_i}{h}\Big) \{I(\varepsilon^\Delta_{i,\alpha}\leq 0) - \alpha\} \over \sumi K\Big(\frac{x-X_i}{h}\Big) f_{\varepsilon_\alpha|X}(0|X_i)}
(\widehat \theta(\alpha) - \theta_0(\alpha))^\top \gamma_\alpha(x)
%(m_{\alpha,\widehat \theta(\alpha)}(x)-m_{\alpha,0}(x))
w(x,\alpha) \, dx \, d\alpha \nonumber \\
&=& {2 \over n} \int_A \int_{R_X}   {\sumi K\Big(\frac{x-X_i}{h}\Big) \{I(\varepsilon^\Delta_{i,\alpha}\leq 0) - \alpha\} \over g_{h,\alpha}(x)} \nonumber \\
&& \qquad \times ( \widehat \theta(\alpha)   -  \theta_0(\alpha)) ^\top \gamma_{\alpha}(x)
w(x,\alpha) \, dx \, d\alpha + o_P((nh^{d/2})^{-1}) \nonumber
\\
&=& 2 \int_A ( \widehat \theta(\alpha)   -  \theta_0(\alpha)) ^\top {1 \over n} \sumi \rho_{h,\alpha}(X_i) \{I(\varepsilon^\Delta_{i,\alpha}\leq 0) - \alpha\} \, d\alpha + o_P((nh^{d/2})^{-1}), \label{lem7claim}
\end{eqnarray}
with $ g_{h,\alpha}(x) = E \big[K\big(\frac{x-X}{h}\big) f_{\varepsilon_\alpha|X}(0|X) \big]$ and
\begin{eqnarray*}
&&\rho_{h,\alpha}(v)  = \int_{R_X}  K\Big(\frac{x-v}{h}\Big) {\gamma_{\alpha}(x)
w(x,\alpha)   \over g_{h,\alpha}(x)}   \, dx.
\end{eqnarray*}
Using the notations $Q_{h,\alpha}(X_i) = \frac{\rho_{h,\alpha}(X_i)}{\sumj \rho_{h,\alpha}(X_j)}$, $\widehat  F_{\varepsilon_\alpha^\Delta}(y) = \sumi Q_{h,\alpha}(X_i) I(\varepsilon_{i,\alpha}^\Delta \le y)$ and $F_{\varepsilon_\alpha^\Delta}(y) = P(\varepsilon_\alpha^\Delta \le y)$, we have that
\begin{eqnarray*}
&& {1 \over n} \sumi \rho_{h,\alpha}(X_i) \{I(\varepsilon^\Delta_{i,\alpha}\leq 0) - \alpha\} \\
&& = \Big[\widehat  F_{\varepsilon_\alpha^\Delta}(0) - \alpha \Big] \Big({1 \over n} \sumi \rho_{h,\alpha}(X_i)\Big) \\
&& = \Big[\widehat  F_{\varepsilon_\alpha^\Delta}(0) - F_{\varepsilon_\alpha^\Delta}(0)\Big] \Big({1 \over n} \sumi \rho_{h,\alpha}(X_i)\Big) + \Big[F_{\varepsilon_\alpha^\Delta}(0) - \alpha \Big] \Big({1 \over n} \sumi \rho_{h,\alpha}(X_i)\Big) \\
&& = O_P(n^{-1/2}) + O_P(n^{-1/2} h^{-d/4}),
\end{eqnarray*}
uniformly in $\alpha \in A$, and hence the statement of the lemma holds because of (\ref{lem7claim}) and (B5).
\end{proof}

\begin{lemma} \label{lem8}
Suppose the assumptions of Theorem \ref{test} are satisfied. Then,
\begin{eqnarray*}
T_{n6} = o_P((nh^{d/2})^{-1}).
\end{eqnarray*}
\end{lemma}
\begin{proof}[Proof of Lemma \ref{lem8}]
The statement of the lemma follows from (B5).

\end{proof}

\begin{lemma} \label{lem9}
Suppose the assumptions of Theorem \ref{test} are satisfied. Then,
\begin{eqnarray*}
n h^{d/2}T_{n7} - b_{h,A} \stackrel{d}{\rightarrow} N(D_A,V_A).\end{eqnarray*}
%provided $nh^d L_n^{-4} \rightarrow \infty$.
\end{lemma}
\begin{proof}[Proof of Lemma \ref{lem9}]  The proof is very similar to the proof of e.g.\ Proposition 1 in
H\"ardle and Mammen (1993). Write
\begin{eqnarray*}
T_{n7} &=& n^{-2} \sum_{i,j} \int_A \int_{R_X} K\Big(\frac{x-X_i}{h}\Big) K\Big(\frac{x-X_j}{h}\Big) \{I(\varepsilon^\Delta_{i,\alpha}\leq 0) - \alpha\} \{I(\varepsilon^\Delta_{j,\alpha}\leq 0) - \alpha\} \\
&& \qquad \times \widehat g_\alpha (x)^{-2} w(x,\alpha) \, dx \, d\alpha,
\end{eqnarray*}
where $\widehat g_\alpha (x)= n^{-1} \sumi K\Big(\frac{x-X_i}{h}\Big) f_{\varepsilon_\alpha|X}(0|X_i)$.  By writing $I(\varepsilon^\Delta_{i,\alpha}\leq 0) - \alpha = \big[I(\varepsilon^\Delta_{i,\alpha}\leq 0) - I(\varepsilon_{i,\alpha}\leq 0)\big] + \big[I(\varepsilon_{i,\alpha}\leq 0) - \alpha\big]$, we can decompose $T_{n7}$ into $T_{n7} = T_{n71} + T_{n72} + 2T_{n73}$.  As in H\"ardle and Mammen (1993), $T_{n73}$ is negligible.  Straightforward calculations show that $T_{n71} = (n h^{d/2})^{-1} (D_A + o_P(1))$.  Indeed,
\begin{eqnarray*}
&& E(T_{n71}|X_1,\ldots,X_n) \\
&& = n^{-2} \sum_{i,j} \int_A \int_{R_X} K\Big(\frac{x-X_i}{h}\Big) K\Big(\frac{x-X_j}{h}\Big) \big[F_{\varepsilon_\alpha|X}(-n^{-1/2} h^{-d/4} \Delta_\alpha(X_i)|X_i) - F_{\varepsilon_\alpha|X}(0|X_i)\big]  \\
&& \qquad \times \big[F_{\varepsilon_\alpha|X}(-n^{-1/2} h^{-d/4} \Delta_\alpha(X_j)|X_j) - F_{\varepsilon_\alpha|X}(0|X_j)\big] \widehat g_\alpha (x)^{-2} w(x,\alpha) \, dx \, d\alpha \\
&& = n^{-2} \sum_{i,j} \int_A \int_{R_X} K\Big(\frac{x-X_i}{h}\Big) K\Big(\frac{x-X_j}{h}\Big) f_{\varepsilon_\alpha|X}(0|X_i) f_{\varepsilon_\alpha|X}(0|X_j) \\
&& \qquad \times n^{-1} h^{-d/2} \Delta_\alpha(X_i) \Delta_\alpha(X_j) \widehat g_\alpha (x)^{-2} w(x,\alpha) \, dx \, d\alpha \: (1+o_P(1)) \\
&& = n^{-1} h^{-d/2} \int_A \int_{R_X} \Delta_\alpha^2(x) w(x,\alpha) \, dx \, d\alpha \: (1+o_P(1)) \\
&& = n^{-1} h^{-d/2} D_A \: (1+o_P(1)).
\end{eqnarray*}
Next, write $T_{n72} = T_{n72a}+ T_{n72b}$ with
\begin{eqnarray*}
T_{n72a} &=& {1 \over n^2} \sumi U_{nii}, \\
T_{n72b} &=& {1 \over n^2} \sum_{i\not = j} U_{nij},
\end{eqnarray*}
where
\begin{eqnarray*} U_{nij}&=& \int_A \int_{R_X}  K\Big(\frac{x-X_i}{h}\Big)K\Big(\frac{x-X_j}{h}\Big) \{I(\varepsilon_{i,\alpha}\leq 0) - \alpha\} \{I(\varepsilon_{j,\alpha}\leq 0) - \alpha\} \\
&& \qquad \times \widehat g_\alpha (x)^{-2} w(x,\alpha) \, dx \, d\alpha.
\end{eqnarray*}
By calculating its mean and variance it can be checked that $n h^{d/2} T_{n72a} = b_{h,A} + o_P(1)$. Thus for the lemma it remains to check that $n h^{d/2} T_{n72b} \stackrel{d}{\rightarrow} N(0,V_A)$. For the proof of this claim one can proceed as in H\"ardle and Mammen (1993) and apply the central limit theorem for U-statistics of de Jong (1987).
For this purpose one has to verify that
$n^2 h^{d} \mbox{Var}(T_{n72b}) \to V_A$, $\max_{1\leq i \leq n} \sum_{j=1}^n \mbox{Var}(U_{nij}) / \mbox{Var}(T_{n72b}) \to 0$ and
$E[T_{n72b}^4] / (\mbox{Var}(T_{n72b}))^2 \linebreak \to 3$. This can be done by straightforward but tedious calculations.
 \end{proof}

\vspace*{.3cm}

\noindent
\begin{proof}[Proof of Theorem \ref{test}]  The theorem follows immediately from Lemmas \ref{lem3}--\ref{lem9}. Lemmas \ref{lem3}--\ref{lem8} imply the negligibility of the terms $T_{n1}$, .., $T_{n6}$. Lemma \ref{lem9} shows the asymptotic normality of  $n h^{d/2}T_{n7}$.
\end{proof}

\section{Proof of Theorem \ref{testboot1}} The theorem can be shown by verification of the conditions of the central limit theorem for U-statistics of de Jong (1987), in the same way as was done in the proof of Lemma \ref{lem9}. The crucial point in the proof is to note that $I(U_i \le \alpha)$ has the same distribution as $I(\varepsilon_{i,\alpha} \le 0)$, and hence the calculations in the proof of Lemma \ref{lem9} go through in this proof.
%Compare also the proof of Lemma \ref{lem9}.

\section*{Acknowledgments}

Research of the first author was prepared within the framework of a subsidy granted to the HSE by the Government of the Russian Federation for the implementation of the Global Competitiveness Program and it was supported by Deutsche Forschungsgemeinschaft
through the Research Training Group RTG 1953. The research of the second author was supported by the European Research Council (2016-2021, Horizon 2020 / ERC grant agreement No.\ 694409), and by IAP research network
grant nr.\ P7/06 of the Belgian government (Belgian Science Policy). %and by the contract ``Projet d'Actions de Recherche Concert\'ees'' (ARC) 11/16-039 of the ``Communaut\'e fran\c{c}aise de Belgique" (granted by the ``Acad\'emie universitaire Louvain'').

\section*{References}

\begin{description}

\item[] Ait-Sahalia, Y., Fan, J. and Peng, H. (2009).
Nonparametric transition-based tests for diffusions.
{\it Journal of the American Statistical Association} {\bf 104} 1102-1116.

\item[] Angrist, J., Chernozhukov, V. and Fernández-Val, I. (2006).
Quantile regression under misspecification, with an application to the U.S.\ wage structure.
{\it Econometrica} {\bf 74} 539-563.

\item [] Bhattacharya,~R.~and~Rao,~R.\
(1976). {\it Normal Approximations and Asymptotic Expansions.}
John Wiley \& Sons, New York.

%\item[] Billingsley, P.\ (1968). {\it Convergence of Probability Measures}. Wiley, New York.

\item[] Bierens, H.J. and Ginther, D. (2001). Integrated conditional moment testing of quantile regression models. {\it Empirical Economics} {\bf 26} 307-324.

\item[] Chaudhuri, P. (1991). Nonparametric estimates of regression quantiles and their local Bahadur representation. {\it Annals of Statistics} {\bf 19} 760-777.

\item[] Conde-Amboage, M., S\'anchez-Sellero, C. and Gonz\'alez-Manteiga, W. (2015). A lack-of-fit test for quantile regression models with high-dimensional covariates. {\it Computational Statistics and Data Analysis} {\bf 88} 128-138.

\item[] De Backer, M., El Ghouch, A. and Van Keilegom, I. (2017). Semiparametric copula quantile regression for complete or censored data. {\it Electronic Journal of Statistics} {\bf 11} 1660-1698. 

\item[] de Jong, P.\ (1987). A central limit theorem for generalized quadratic forms. {\it Probability Theory and Related Fields} {\bf 75} 261-277.

\item[] Dette, H.\ and Sprekelsen, I. (2004). Some comments on specification tests in nonparametric absolutely regular processes. {\it Journal of Time Series Analysis} {\bf 25} 159-172.

\item[] El Ghouch, A.\  and Van Keilegom, I.\ (2009). Local linear quantile regression with dependent censored data. {\it Statistica Sinica} {\bf 19} 1621-1640.

\item[] Fan, J., Zhang, C. and Zhang, J. (2001).
Generalized likelihood ratio statistics and Wilks phenomenon.
{\it Annals of Statistics} {\bf 29} 153-193.

\item[] Gao, J. and Hong, Y.\ (2008). Central limit theorems for generalized U-statistics with applications in nonparametric specification. {\it  Journal of Nonparametric Statistics} {\bf  20} 61-76.

\item[] Gonz\'alez-Manteiga, W.\ and Cao-Abad, R.\ (1993).
Testing the hypothesis of a general linear model
using nonparametric regression estimation. {\it Test} {\bf 2} 161-188.

\item[] Guerre, E. and Lavergne, P. (2002). Optimal minimax rates for nonparametric specification testing in regression models. {\it Econometric Theory} {\bf 18} 1139-1171.

\item[] Guerre, E.\ and Sabbah, C. (2012). Uniform bias study and Bahadur representation for local polynomial estimators of the conditional quantile function. {\it Econometric Theory} {\bf 28} 87-129.

\item[] Haag, B. (2008).
Non-parametric regression tests using dimension reduction techniques. {\it Scandinavian Journal of Statistics}
{\bf 35} 719-738.

\item[] H\"{a}rdle, W.\ and Mammen, E.\ (1993). Testing parametric
versus nonparametric regression. {\it Annals of Statistics} {\bf 21} 1926-1947.

\item[] He, X. and Ng, P. (1999). Quantile splines with several covariates. {\it Journal of Statistical
Planning and Inference} {\bf 75} 343-352.

\item[] He, X., Ng, P. and Portnoy, S. (1998). Bivariate quantile smoothing splines. {\it Journal
of the Royal Statistical Society - Series B} {\bf 60} 537-550.

\item[] He, X. and Zhu, L.-X. (2003).
A lack of fit test for quantile regression.
{\it Journal of the American Statistical Association} {\bf 98} 1013-1022.

\item[] Hjellvik, V., Yao, Q.\ and Tj\o stheim, D.\ (1998). Linearity testing using local polynomial
approximation. {\it Journal of Statistical Planning and Inference} {\bf 68} 295-321.

\item[] Hoderlein, S. and Mammen, E. (2009). Identification and estimation of local average derivatives in non-separable models without monotonicity. {\it Econometrics Journal} {\bf 12} 1-25.

\item[] Hong, S.Y. (2003). Bahadur representation and its applications for local polynomial estimates in non-parametric M-regression. {\it Journal of Nonparametric Statistics}  {\bf15} 237-251.

\item[] Horowitz, J.L. and Spokoiny, V.G. (2002).
An adaptive, rate-optimal test of linearity for median regression models.
{\it Journal of the American Statistical Association} {\bf 97} 822-835.

\item[] Ingster, Y.I. (1993).
Asymptotically minimax hypothesis testing for nonparametric alternatives I, II, III.
{\it Math.\ Methods of Statistics} {\bf 2} 85-114, 171-189, 249-268.

\item[] Koenker, R. (2005).
{\it Quantile Regression.}
Cambridge University Press.

\item[] Koenker, R. and Machado, J.A.F. (1999).
Goodness of fit and related inference processes for quantile regression.
{\it Journal of the American Statistical Association} {\bf 94} 1296-1310.

\item[] Koenker, R. and Xiao, Z. J. (2002).
Inference on the quantile regression process.
{\it Econometrica} {\bf 70} 1583-1612.

\item[] Kong, E., Linton, O. and Xia, Y. (2010).
Uniform Bahadur representation for local polynomial estimates of M-regression and its application to the additive model.
{\it Econometric Theory} {\bf 26} 1529-1564.

\item[] Kreiss, J.P., Neumann, M.H.\ and Yao, Q.\ (2008).
Bootstrap tests for simple structures in nonparametric time series regression.
{\it Statistics and its Interface} {\bf 1} 367-380.

\item[] Lee, K.L. and Lee, E.R. (2008).
Kernel methods for estimating derivatives of conditional quantiles.
{\it Journal of the Korean Statistical Society} {\bf 37} 365-373.

\item[] Leucht, A. (2012).
Degenerate $U$- and $V$-statistics under weak dependence: Asymptotic theory and bootstrap consistency.
{\it Bernoulli} {\bf 18} 552-585.

\item[] Li, Q. and Racine, J.S. (2008).
Nonparametric estimation of conditional CDF and quantile functions with mixed categorical and continuous data.
{\it Journal of Business \& Economic Statistics} {\bf 26} 423-434.

\item[] Rothe, C. and Wied, D. (2013).
Misspecification testing in a class of conditional distributional models.
{\it Journal of the American Statistical Association} {\bf 108} 314-324.

\item[] Su, L. and White, H.L. (2012).
Conditional independence specification testing for dependent processes with local polynomial quantile regression.
{\it Advances in Econometrics} {\bf 29} 355-434.

\item[] Truong, Y.K. (1989).
Asymptotic properties of kernel estimators based on local medians.
{\it Annals of Statistics } {\bf 17} 606-617.

\item[] Volgushev, S., Birke, M., Dette, H. and Neumeyer, N. (2013).
Significance testing in quantile regression.
{\it Electronic Journal of Statistics} {\bf 7} 105-145.

\item[]   Yu, K.  and Jones,  M. C.   (1998).
Local Linear Quantile Regression.
{\it Journal of the American Statistical Association} {\bf 93} 228-237.

\item[] Zheng, J.X. (1996).
A consistent test of a functional form via nonparametric estimation techniques.
{\it Journal of Econometrics} {\bf 75} 263-289.

\item[] Zheng, J.X. (1998).
A consistent nonparametric test of parametric models under conditional quantile regressions.
{\it Econometric Theory } {\bf 14} 223-238.

\end{description}

\end{document}